\documentclass[reqno]{amsart}

\usepackage[foot]{amsaddr}

\makeatletter
\@namedef{subjclassname@2020}{\textup{2020} Mathematics Subject Classification}
\makeatother

\usepackage{amssymb} 
\usepackage{amsmath, amsthm} 
\usepackage{graphicx}
\usepackage[top=3cm, bottom=3cm, left=3cm, right=3cm]{geometry}

\usepackage{tikz}
\usetikzlibrary{intersections}
\usetikzlibrary{decorations.markings}
\usetikzlibrary{arrows}
\usetikzlibrary{positioning}

\usepackage{enumitem}
\usepackage{cleveref} 

\usepackage[T1]{fontenc}
\DeclareMathAlphabet{\mathpzc}{OT1}{pzc}{m}{it}
\usepackage{yfonts}
\usepackage[sans]{dsfont}
\usepackage{txfonts}
\usepackage[mathscr]{euscript}
\usepackage{bbm}

\newtheorem{thm}{Theorem}
\newtheorem{thm2}{Theorem}
\newtheorem{cor}[thm]{Corollary}
\newtheorem{lem}[thm]{Lemma}
\newtheorem{prop}[thm]{Proposition}
\theoremstyle{definition}		

\newcommand{\set}[1]{\mathpzc{#1}}	

\newcommand{\itr}{\mathbb Z}			
\newcommand{\rat}{\mathbb Q}			
\newcommand{\nat}{\mathbb N}			

\newcommand{\ff}[1]{{\mathbb F}_{#1}}		
\newcommand{\ffx}[1]{\ff{#1}[X]}		

\DeclareMathOperator{\Nm}{N}

\DeclareMathOperator{\Min}{Min}			

\newcommand{\rhombicfl}[1]{\left\lfloor #1 \right\rfloor_{\diamondsuit}}

\DeclareMathOperator{\MCD}{MCD}			
\newcommand{\qf}[1]{\rat(\sqrt{-#1})}

\DeclareMathOperator{\cost}{cost}		

\newcommand{\X}{\set X}				

\begin{document}
\title[More on\dots]{More on the number of distinct values of a class of functions}

\author[R.S. Coulter]{Robert S. Coulter}

\address[R.S. Coulter]{Department of Mathematical Sciences, University of Delaware,
Newark, DE 19716, United States of America.}

\author[S. Senger]{Steven Senger}

\address[S. Senger]{Department of Mathematics, Missouri State University, MO 65897, United States of America.}

\email[R.S. Coulter]{coulter@udel.edu}
\email[S. Senger]{stevensenger@missouristate.edu (corresponding author)}

\thanks{Orcid IDs: (R.S. Coulter) 0000-0002-1546-8779,
(S. Senger) 0000-0003-2912-4464}

\thanks{The authors gratefully acknowledge a bequest from the Estate of Fransisco Javier ``Pancho" Sayas, which partially supported this research.}

\dedicatory{Dedicated to the memory of Pancho (1968--2019)}


\begin{abstract}
In a previous article the authors determined the best-known
upper bound for the cardinality of the
image set for several classes of functions, including planar functions. 
Here, we show that the upper bound cannot be tight for planar functions over
finite fields.
This follows from a more general result proving that the upper bound cannot be 
tight for a much larger class of functions over an abelian group of order
$y^n$ with $n>1$. Moreover, the tightness of the upper bound for the larger
class of functions is equivalent to the existence of planar difference sets.

To obtain better upper bounds, we first completely
resolve an optimization problem involving the partitioning of a number into
triangular parts.
Our solution, which is algorithmic and constructive,
allows us to determine tight upper bounds provided
the relevant parameters are given explicitly.
We also provide a suite of upper bounds which can be applied across
a range of parameters.
These are established via a well-studied Diophantine equation 
and are related to class numbers of quadratic number fields.
\end{abstract}

\maketitle

\section{Introduction and history}

Throughout, we use $A$ to denote a set of cardinality $q$ and $G$ to denote
an arbitrary abelian group of order $q$, written
additively with identity $0$. In order to avoid some trivialities, we assume $q>7$ for all that follows.
We use $G^\star$ to denote the nonzero
elements of the group. We will also have cause to discuss finite fields,
which we denote by $\ff{q}$ with $q=p^n$, $p$ a prime and $n\in\nat$.

Let $f$ be a function on $A$.
Following notation of Carlitz \cite{carlitz55}, we define the following terms.
\begin{itemize}

\item $V(f)$ is the cardinality of the image set of $f$. That is,
$V(f)=\#\{f(x)\,:\, x\in A\}$.

\item $f$ is a {\em permutation} if $V(f)=q$.
Moreover, if $A=\ff{q}$, we talk of a {\em permutation polynomial (PP)} over
$\ff{q}$,
as any function over $\ff{q}$ can be represented by a polynomial in $\ffx{q}$.

\item When $A=G$, for any $a\in G^\star$, the {\em differential operator of $f$ in the
direction of $a$} (or simply non-trivial differential operator) is the function
$\Delta_{f,a}$ defined by $\Delta_{f,a}: x\mapsto f(x+a) - f(x)$.

\item For any integer $k\ge 2$, we define
$N_k(f)$ to be the number of $k$-tuples of distinct inputs with the same
output under $f$. That is,
\[N_k(f) = \#\{(x_1, x_2,\dots x_k)\,:\, x_i\in A \land x_i\ne x_j \land
f(x_1)=f(x_2)=\dots=f(x_k)\}.\]
Note that order matters, so that $N_k(f)$ is necessarily divisible by $k!$.

\end{itemize}
The expected value of $N_2(f)$ for $f\in\ffx{q}$ was shown to be $q-1$ in
\cite{coulter14c}. In fact, this follows at once from the the following short,
but far more general, argument:
Given any function $f$ on $A$, there are
${q \choose k} k!$ different $k$-tuples, $(x_1, x_2, \dots, x_k).$
The probability that $f(x_i)=f(x_j)$ for all relevant $i$ and $j$ is easily
seen to be $q^{-(k-1)}.$ This gives that the expected value of $N_k(f)$ is
${q \choose k} \frac{k!}{q^{k-1}}.$ Clearly, the expected value of $N_2(f)$ is
$q-1$. Note, that the
expected value for $N_k$ for $k\ge 3$ is never an integer. To see this, suppose
otherwise. Then as $\gcd(q,q-1)=1$, we would have
$(q-2)(q-3)\cdots (q-(k-1))/q^{k-2}\in\itr$, but this is a product of $k-2$
rationals all less than 1.

There are several classes of functions we will be particularly
focused on in
this paper. We follow the notation previously established in \cite{coulter14c}.
\begin{itemize}
\item A function $f$ on $G$ is {\em planar} if every non-trivial differential operator of $f$
is a permutation on $G$.
\item $f$ is said to belong to class $C_3(G)$ if every
non-trivial differential operator, $\Delta_{f,a}$ with $a\in G^\star,$ has a
unique zero.
\item A function $f$ on $A$ belongs to class $C_4(A)$ if $N_2(f)=q-1$. 
Note that $C_4$ is empty if $q$ is even.
\end{itemize}
Class $C_3(G)$ is a proper subclass of class $C_4(G)$, and the class of all
planar functions over $G$ is a proper subclass of class $C_3(G)$. These
inclusions were shown in \cite{coulter14c}.

There has been much interest in bounding the size of the image sets of planar
functions. (Further motivation for studying planar functions is given in
\Cref{motivation}.)
In \cite{coulter14c}, the authors provided bounds on the image sets
of functions in terms of $N_k.$
In particular, for functions in the class $C_4,$ we obtained
\begin{equation} \label{bounds}
\frac{q+1}{2} \le V(f) \le q - \frac{2(q-1)}{1+\sqrt{4q-3}}
= q - \frac{-1+\sqrt{4q-3}}{2}.
\end{equation}
While many planar functions are known to attain the lower bound, there are no
known examples of functions in $C_4$ that meet the upper bound.
Indeed we have always suspected that the bound is not tight for functions in
$C_3$, and most probably $C_4$ as well.
As we noted in \cite{coulter14c}, for the upper bound to be achieved $(q-1)/2$
needs to be a triangular number, so that $t(t-1)=q-1$ for some integer $t$.
Moreover, there are $t$ distinct elements $x_1,\ldots,x_t \in A$ for which 
$f(x_i)=f(x_j)$, with $f$ injective on $A\setminus\{x_i\}$.
This was recently re-visited for planar functions specifically as part of a
more wide-ranging paper of K\"olsch and Polujan, see \cite{kolschtoappear},
Proposition 6.1.
This structure, and similar scenarios, will be the focus of this article in
which we will improve the upper bound for functions in $C_4$.
We first will establish the following result.
\begin{thm} \label{thm1}
The following statements hold:
\begin{enumerate}[label = (\roman*)]
\item Any function in $C_3(G)$ that meets the
upper bound given in \eqref{bounds} must describe a projective plane of order $k$, where $k^2+k+1=q.$ Indeed, any such function is equivalent to an abelian
planar difference set of order $q$.
\item Let $q=y^n$ for natural numbers $y,n$ with $y>1$.
If $n>1$ and $q\ne 343$, or there exists
a prime $p\not\equiv 1\bmod 6$ dividing $q$, then there is no function in
$C_4(A)$ that meets the upper bound of \eqref{bounds}.
\item If $q=343,$ then there is no function in $C_3(G)$ that meets the upper
bound of \eqref{bounds}.
\item There is no planar function over a finite field that meets the upper
bound of \eqref{bounds}.
\end{enumerate}
\end{thm}
As a consequence of this result, we return to our analysis in \cite{coulter14c},
and attack the problem we identified there that lies at the heart of 
determining the upper bound. This problem, which is connected to sums of
triangular numbers, is completely resolved (at least in an algorithmic sense)
in \Cref{CostSection}.
In fact, our solution to the problem provides tight upper bounds
for the image set of functions not just in $C_4(A)$, but for functions defined
on any finite set with any feasible value for $N_2$, constructively through the
algorithm. The explicit
bounds are quite technical to state and (obviously) algorithmic in nature,
so we omit them from this introduction;
see \Cref{costProp} and \Cref{cobbleCost} in particular.

While the algorithmic solution gives us a way to provide explicit tight upper
bounds for $V(f)$ for any specified $q$ and $N_2$, we also determine upper bounds
on functions in $C_4$
that apply for general $q$, though for effectiveness we still separate $q$ into
various sub-cases.
When $q$ is a square, we obtain the following new upper bound.
\setcounter{thm2}{1} 
\begin{thm} \label{squarebound}
Fix a nonsquare integer $y>1$ and set $q=y^{2^j}$ for some $j\ge 1$.
Suppose $f\in C_4(A)$.
Then
\begin{equation*}
V(f) \le q + j - \sum_{i=0}^{j-1} y^{2^i} -
\begin{cases}
7 &\text{ if $y=5$,}\\
15 &\text{ if $y=11$, and}\\
\frac{-1+\sqrt{4y-3}}{2} &\text{ otherwise.}
\end{cases}
\end{equation*}
\end{thm}
Now set $D_i:=8i+3$ when $i\geq 0.$ When $q$ is a non-square power, we obtain the following bound.
\begin{thm} \label{newgeneralbound}
Let $q=y^n$ for odd natural numbers $y,n>1$ and let $f\in C_4(A)$.
For $i=0,1,2,7,10,13,16,18$, set
$\beta_i=0,1,2,4,4,6,6,7$, respectively.
Then 
\begin{equation} \label{gencosteq}
V(f) \le q - \left(\frac{\sqrt{4q-D_i} - 1}{2}\right) - \beta_i, 
\end{equation}
where
\begin{equation*}
i = 
\begin{cases}
0	&\text{ if $y^n=7^3$,}\\
1	&\text{ if $y^n=3^5$,}\\
2	&\text{ if $y^n=5^7$,}\\
7	&\text{ if $y^n\in \{3^3, 5^3, 41^3, 591^3\}$,}\\
10	&\text{ if $y^n=153^3$,}\\
13	&\text{ if $n=3^a$ for $a\ge 1$ and $y\notin\{3,5,7,41,153,591\}$,}\\
13	&\text{ if $n=3^a$ for $a\ge 2$ and $y\in\{3,5,7,41,153,591\}$,}\\
16	&\text{ if $n=5^a$ for $a\ge 1$ and $y\ne 3$,}\\
16	&\text{ if $n=5^a$ for $a\ge 2$ and $y=3$, and}\\
18	&\text{ otherwise.}\\
\end{cases}
\end{equation*}
This bound is tight in each of the exceptions given in the above list.
\end{thm}
To be explicit, for all non-trivial odd powers $q$ not in the list of exceptions
given in \Cref{newgeneralbound}, we get the bound
\begin{equation*}
V(f) \le q - 7 - \left(\frac{\sqrt{4q-147} - 1}{2}\right).
\end{equation*}
We also give new upper bounds for $V(f)$ when $q$ is a non-square prime power.
To state it, we first define for any odd number $m\ge 3$, the
{\em modified class divisor} function $\MCD(m)$ to be the smallest $D_i$ for
which $\gcd(m,h(-D_i))>1$, where $h(-D_i)$ is the class number of 
$\qf{D_i}$.
\begin{thm} \label{secondnewbound}
Let $q=p^n$ with $p$ an odd prime and where $n\ge 11$ is odd and
$\gcd(n,15)=1$. Additionally, assume 
$q\notin\{3^{11}, 5^{11}, 5^{13}\}$.
If $f\in C_4(A)$, then
\begin{equation} \label{finalnewbound}
V(f)
\le q + 1 - \left(\frac{\sqrt{4q-D_j} + \sqrt{4j+1}}{2}\right),
\end{equation}
where 
\begin{equation*}
D_j =
\begin{cases}
\Min\left\{\MCD(n),\MCD(n-2)p^2,107p^4\right\} &\text{ if $n\equiv 1\bmod 3$, or}\\
\Min\left\{\MCD(n),107p^2\right\} &\text{ if $n\equiv 2\bmod 3$.}
\end{cases}
\end{equation*}
\end{thm}
The remaining non-square prime power cases are also dealt with, see 
\Cref{nonsquareprime}, especially \Cref{boundforn=7} for the
$n=7$ case.
All of the above bounds apply for planar functions also.

\section{Motivation}\label{motivation}

Planar functions were introduced by Dembowski and Ostrom \cite{dembowski68} in
1968 to construct projective planes with specific automorphism groups, but
have been studied more recently in relation to cryptographic functions as they
offer the best resistance over fields of odd characteristic against differential
cryptanalysis when used in a substitution box. Ding and Yuan \cite{ding06}
also used them to disprove a long-standing conjecture concerning skew Hadamard
difference sets.
A well-known conjecture in the area is the Dembowski-Ostrom (DO) conjecture,
which states that all planar functions over $\ff{q}$, $q=p^n$ with $p$ an odd
prime, are necessarily equivalent to one of the form
$$\sum a_{ij} X^{p^i+p^j}.$$
Here, equivalence is defined as follows: $f,h\in\ffx{q}$ are equivalent if
there exist three linear operators $L,M,N$, with $L,M$ nonsingular, and
a constant $c\in\ff{q}$,  satisfying
$$L(f(X))\equiv h(M(X))+N(X)+c \bmod (X^q-X).$$
The conjecture was proved to be correct over prime fields nearly simultaneously
in three independent papers by Gluck \cite{gluck90a}, Hiramine
\cite{hiramine89}, and R\' onyai and Sz\" onyi \cite{ronyai89}. Thus, over
prime fields, the only planar functions are precisely the quadratics, and these
form a single equivalence class.
The conjecture is false over fields of characteristic 3. This was shown by
Coulter and Matthews \cite{coulter97b}, who provided an infinite class of
counterexamples, the smallest being $X^{14}$ over $\ff{3^4}$.
The conjecture remains open for 
characteristics larger than 3, and also for fields of order $p^2$ and $p^3$.
If the DO conjecture is true for fields of order $p^2$ and $p^3$, then all
planar functions have been classified in those cases also. Specifically,
they must be equivalent to $X^2$ over $\ff{p^2}$, and to either
$X^2$ or $X^{p+1}$ over $\ff{p^3}$. These classifications follow from a
combination of results from Knuth \cite{knuth65},
Menichetti \cite{menichetti77}, and Coulter and Henderson \cite{coulter08}.

With regard to their cryptographic applications, define 
{\em the differential uniformity (DU)} of a function $f$ on $G$ as
the maximum number of preimages of any image of any differential operator 
$\Delta_{f,a}$ of $f$. It has been known for many years the smaller the DU of
a function, the better the resistance the function offers against differential
cryptanalysis when used in an S-box. Planar functions have (the obviously
optimal) DU of 1. For these reasons, it
would be preferable to have planar functions which are also PPs.
This is, however, impossible. This is easily seen from the fact every
differential operator needs to have a zero.
The question of just how far a planar function is from being a PP was the motivation behind \cite{coulter14c}, producing the upper bound in \eqref{bounds}.
Our own understanding of our bound, along with the paper \cite{kolschtoappear},
formed the impetus for us to return to, and improve, that bound in the current
paper.

\section{Correspondence between the upper bound and a projective plane of order $t-1$}

While our primary motivation in studying the upper bound was to improve it,
we now establish a second, somewhat surprising
reason one might be interested in the upper bound.
Recall that, a subset $S\subset G$
of cardinality $k$ forms a {\em $(q,k,\lambda)$-difference set}, or
$(q,k,\lambda)$-DS, if every $g\in G^\star$, can be expressed as a difference, $g=x-y,$ with $x,y\in S,$
in exactly $\lambda$ ways.
The special case where $\lambda=1$ is known as a {\em simple or planar} DS.
In particular, any $(k^2+k+1,k+1,1)$-DS is equivalent to a projective plane of
order $k$. We refer the reader to Jungnickel and Schmidt \cite{jungnickel97}
or van Lint and Wilson \cite{blint01} for further information regarding
difference sets. For results regarding planar difference sets specifically, see
Gordon \cite{gordon94}.

Let $f\in C_3(G)$ meet the upper bound in \eqref{bounds} and let $\alpha$ be
the unique image that occurs $t$ times. Recall that $t(t-1)=q-1$.
Set $S=\{x_i\}$ to be the set of pre-images of $\alpha.$
Since $f\in C_3(G)$, each difference operator has exactly one
zero. More specifically, for each shift $a\ne 0$, we have precisely one solution
to $f(x+a)=f(x)$ where $x+a,x\in S$. Consequently, for every $a\in G^\star$,
there exists a unique pair $(x_i,x_j)$ satisfying $x_i-x_j=a$.
Thus, $S$ forms a $(q,t,1)$-DS. Now set $k=t-1$.
Then $q=t^2-t+1=k^2+k+1$, so that $S$ is in fact a $(k^2+k+1,k+1,1)$-DS.
Conversely, if you have a planar difference set in $G$, then to construct a
function $f\in C_3(G)$ that meets the upper bound in \eqref{bounds} is
straightforward: one simply fixes one image
for all elements of the difference set, and distinct images for all remaining
elements of $G$ not in the difference set.
This proves \Cref{thm1} (i).

Note that since a planar function over $G$ is equivalent to a specific type of
projective plane of order $q$, we see that one reaching the upper bound would
actually be associated with two distinct projective planes, one of order $q$ and
one of order $t-1$, with $t-1$ dividing $q-1$. At the end of the next section,
we will show this is impossible for planar functions over finite fields.

\section{Showing the upper bound is nigh impossible} \label{step0section}

As we've noted, for a function in $C_4(A)$ to attain the upper bound in
\eqref{bounds}, there needs to exist an integer $t$ satisfying
$t(t-1)=q-1$.

\begin{prop}\label{reductions}
Let $q=y^n,$ for natural numbers $y$ and $n,$ with $y>1$.
If $f\in C_4(A)$ meets the upper bound in \eqref{bounds}, then $n$ is odd and every prime $p$
dividing $y$ satisfies $p \equiv 1 \bmod 6.$
\end{prop}
\begin{proof}
Note that since $f$ attains the upper bound from \eqref{bounds}, there must exist a $t\in\nat$ satisfying
$t(t-1)=q-1$. This immediately forces $q$ to be odd.

To prove $n$ must be odd, by way of contradiction, suppose $n=2m$ for some natural number $m.$ Then $t(t-1)=(y^m-1)(y^m+1)$. This means there exist $a,b\in \nat,$ where $a(a+1)=b(b+2)$.
This implies that
$b<a<a+1<b+2$, so that $a, a+1$ are two integers lying strictly between the
integers $b$ and $b+2,$ a contradiction.

Next, assume $p$ is a prime dividing $y$. As $t^2-t+1-q=0$, working modulo $p$, we find
the discriminant of the resulting quadratic in $t$ is $-3$. For there to be a solution, the discriminant must be a square mod $p$.
Using well-known results on the Legendre symbol, this forces
$p\equiv 1\bmod 6$.
\end{proof}

For the remainder of this section, we restrict our attention to the case where
$q=y^n,$ for natural numbers $y,n>1.$
An equivalent version of our problem is that there needs to be an integer $x$
satisfying $x^2+x+1=q$.
Part of \Cref{reductions} reduces our problem to determining
non-trivial integer solutions $(x,y,n)$ to the equation
\begin{equation}\label{Nagell}
x^2+x+1 = y^n
\end{equation}
with $n\ge 3$ odd. (Clearly we always have the trivial solution $(0,1,n)$ if
we allow $y=1$. Since there is no non-trivial differential operator over the
trivial group, we need not consider this case.) 

There are a number of remarkable histories connected to this problem.
In 1916, Thue \cite{thu16} proved that for fixed integers $a,b,c,d,n$ with
$ad(b^2-4ac)\ne 0$ and $n\ge 3$, the equation 
\begin{equation} \label{EQ1}
ax^2+bx+c=dy^n
\end{equation}
has only a finite number of non-trivial integer solutions in $(x,y)$.
Soon after, Nagell \cite{nag21} showed that \eqref{Nagell} has only
trivial solutions unless $n$ is a power of $3$ which, though dealing with a
more restricted setting than Thue, goes much further in that this deals with
all $n\ge 3$ at once.
Indeed, Nagell's result shows there are only finitely many non-trivial
solutions to \eqref{Nagell}, since his result effectively reduces the problem
to finding solutions for $n=3$ only and Thue's result does the rest.

The $n=3$ case can be dealt with through standard modern techniques from 
abstract algebra, along with some more number theory stemming from Thue.
Suppose we have a solution to \eqref{Nagell}.
Then the discriminant must be an integer, and so we see that
a solution to (\ref{Nagell}) is equivalent to a non-trivial integer solution to the equation
\begin{equation} \label{nag2}
z^2 + 3 = 4y^3.
\end{equation}
Set $\sigma=\sqrt{-3}$ and $\omega = (1+\sigma)/2$. Note that $\omega$ is a
primitive 6th root of unity and $\omega^2=\omega-1$.
We now work in the UFD $\itr(\omega)$, where the only units are powers of
$\omega$.
Recall that in $\itr(\omega)$ all elements take the form
$a+b\sigma$, with $a,b\in\itr\left(\frac{1}{2}\right)$ and $2(b-a)$ even.
Further, $\Nm(a+b\sigma)=a^2+3b^2$ acts as a multiplicative Euclidean norm on
$\itr(\omega)$.
Suppose we have a solution to (\ref{nag2}).
Clearly $z=2t+1$ for some integer $t$. Then
\begin{align*}
y^3 &= \frac{1}{4} (z+\sigma)(z-\sigma)\\
&= \left(\frac{2t+1+\sigma}{2}\right) \left(\frac{2t+1-\sigma}{2}\right)\\
&=(t+\omega)(t+1-\omega) = \alpha\beta.
\end{align*}
Since we have a Euclidean norm, we can talk of greatest common divisors.
Let $d\in\gcd(\alpha,\beta)$.
Now $\alpha+\beta=2t+1=z$, while $\alpha-\beta=2\omega-1=\sigma$.
It is easily observed that $\sigma$ is a prime in $\itr(\omega)$ as
$\Nm(\sigma)=3$ and we are operating within a UFD.
Since $d$ divides both $\alpha+\beta$ and $\alpha-\beta$, ignoring units $d$
is either a unit or $\sigma$.

If $d=\sigma$, then $\sigma|z$. Taking the norm, we get $\Nm(\sigma)=3|z^2$,
implying $3|z$.
Returning to (\ref{nag2}), we find $3|y$, and now working modulo 9 we get
$3\equiv 0\bmod 9$, a contradiction.

So $d$ must be a unit, and since $\alpha\beta$ is a cube, we must have that
that both $\alpha$ and $\beta$ are cubes also.
Thus, $t+\omega = \omega^k(a+b\omega)^3$ for some
$a,b\in\itr$ and with $k\in\{0,1,2\}$. Focusing on the coefficient of
$\omega$ in the expansion we arrive at one of three possibilities based on
$k$:
\begin{enumerate}
\renewcommand{\labelenumi}{$k = $ \arabic{enumi} :}
\setcounter{enumi}{-1}
\item $1=3a^2b+ab^2=ab(3a+b)$
\item $1=a^3+3a^2b-b^3$
\item $1=a^3-3ab^2-b^3$
\end{enumerate}
Clearly there is no solution to the $k=0$ case, while
the latter two equations are equivalent, in that $(a,b)$ is a solution to
the $k=1$ case if and only if $(-b,-a)$ is a solution to the $k=2$ case.
We deal with the $k=2$ case.
In 1933, Skolem \cite{sko33} gave an effective method for solving cubic
Diophantine equations with positive discriminant (which this equation has) and
specifically referenced the equation for $k=2$, conjecturing that his methods
would solve it completely. While he did not carry out this work, Ljunggren did
in 1943 \cite{lju43}. He showed that the only solutions are given by
$(a,b) \in\{ (1,0), (0,-1), (-1,1), (1,-3), (-3,2), (2,1)\}$.
These pairs produce just two solutions for \eqref{nag2},
namely $z=1$ and $z=37$.
Translating to the original equation \eqref{Nagell}, we find
the only non-trivial solution is given by $(x,y,n)=(18,7,3)$.

It should be mentioned that the equations for $k=1$ and $k=2$ involve
irreducible homogeneous functions in $a,b$. Thue \cite{thu09} showed in 1909
that for any irreducible homogeneous polynomial $f\in\itr[X,Y]$ of degree at
least 3, and any integer $c$, there are only finitely many integer solutions to 
the equation $f(x,y)=c$.
Such an equation is often nowadays referred to as a Thue-type equation.
A modern treatment of the $k=1$ or $k=2$ equation would not rely on the
methods outlined by Skolem and employed by Ljunggren, as there are now much
more effective computational methods for solving such equations,
see Bilu and Hanrot \cite{bilu96} for example.
The Magma algebra package \cite{bos97}, using the methods of Bilu and Hanrot,
almost instantly yields the same solution set for the $k=2$ equation as
that determined by Ljunggren.

From the above we conclude that any function in $C_4(A)$ with $q=y^n$,
$y,n>1$, meeting the upper bound in \eqref{bounds} would necessarily require
$q=343$. If $f\in C_3(G)$, then it would also imply the existence of a
projective plane of order 18 by \Cref{thm1} (i) which would come from a planar
difference set of order $k=18$. However, it was shown by Jungnickel and Vedder in
\cite{jungnickel84}, Corollary 3.4, that a necessary condition
for the existence of a planar difference set of even order $k$ is
$k\equiv 0\bmod 4$. Thus, Jungnickel and Vedder's result implies no such
function exists.  This completes the proof of \Cref{thm1} (iii).

Finally, to prove \Cref{thm1} (iv), if $f\in\ffx{q}$ meets the upper bound
in \eqref{bounds}, then \Cref{thm1} (ii) and (iii) imply that $q$ must be
a prime. However, as was mentioned in \Cref{motivation}, it has been shown that
any planar function over a prime field must be quadratic, and every
quadratic over a prime field $\ff{q}$ produces an image set of size $(q+1)/2$,
which is the lower bound of \eqref{bounds}. This proves \Cref{thm1} (iv).

\section{First step towards improving the upper bound}\label{sec5}

The upper bound in \eqref{bounds} could only possibly be attained if there is a unique element with multiple pre-images, and the other elements of the image
set each have a unique pre-image. However, \Cref{thm1} precludes this possibility for functions $f\in C_4(A)$ with $q\neq 343.$
We therefore wish to improve the bound, possibly so that it is tight.
To begin, let us recall a discussion from our earlier paper \cite{coulter14c}:
\begin{quote}
Let $T_r=r(r-1)$ for any $r\in\nat$, and fix $k\in\nat$.
By a {\em triangular sum of length $l$ for $k$} we mean any instance of the
equation
\begin{equation*}
k = \sum_{i=1}^l T_{r_i},
\end{equation*}
where $r_1\ge r_2\ge\cdots\ge r_l$.
The {\em weight} of a given triangular sum is given by $-l+(\sum_{i=1}^l r_i)$.
Given $k$, we define $B_k$ to be the smallest weight among all triangular sums
for $k$.
\end{quote}
As we noted at the time, it is clear that when $k=T_u$, $B_k=u-1$. While
Gauss proved there exists a triangular sum for any $k$ with length at most
3, it may not be the case that such an instance will provide the value
of $B_k$. As we underlined in \cite{coulter14c}, the connection to the
upper bound is clear: if $N_2(f)=2k$, then $V(f)\le q-B_k$, with equality
always possible.
Based on this discussion, to improve the bound we first want to understand
exactly how one determines $B_k$.


We now define a helpful notion on equivalence relations that
will allow us to bound the size of the image set of a function.
Let $\sim$ be any equivalence relation defined on a set $A.$ Set $w(x,x')=1/s$ when distinct $x$ and $x'$ are in an equivalence class of size $s,$ and zero otherwise. We define the
{\it total cost} of the relation $\sim$ as
\[\cost(\sim) = \sum_{(x,x')\in A} w(x,x').\]
As a direct consequence of this definition, if $\sim$ has $k$ equivalence classes of sizes $s_1,s_2,\ldots,s_k,$
we will find
\[\cost(\sim) = \sum_{i=1}^k (s_i - 1).\]
Notice that any equivalence class consisting of only a single element will contribute zero to the total cost. More generally, the cost of an equivalence class
of size $s$ will be $s-1.$ 
We use $||\sim ||$ to denote the number of equivalent ordered pairs of distinct elements, which we call the {\it size} of the equivalence relation. Thus, 
\[||\sim|| = \sum_{i=1}^k s_i(s_i-1).\]
Given $n\in \nat$, we define the {\it rhombic floor} of $n$, denoted by $\rhombicfl{n}$ to be the largest natural number $t$ satisfying $t(t-1)\leq n$. That is,
\[\rhombicfl{n}=\max\{t\in \nat \,:\, t(t-1)\leq n\}.\]

This definition may look strange for a number of reasons, so we pause for a moment to explain why we have chosen this peculiar object for our analyses. For our purposes, the significance of the rhombic floor of $n$ is that it is the largest size of any equivalence class in an equivalence relation with $n$ equivalences. This will give us a kind of division algorithm that will allow us to analyze the minimal possible cost of $n$. Moreover, it is clear that any rhombic number is just twice a triangular number.
(Rhombic numbers have also been called oblong numbers, among other terms, but
we are emphasizing that they are twice a triangular number to reiteratate the
connection to the triangular number problem stated above.)
While this may seem to be a superficial complication, we chose this for expositional reasons, as it will save us from crowding the expressions that will soon follow with multiple divisions by two.

Next, when $n$ is an even natural number, we will use $\cost(n,k)$ to denote the minimum cost of any equivalence relation of size $n$ with exactly $k \leq \frac{n}{2}$ non-singleton equivalence classes. If no such equivalence relation can exist (for example, if $2k>n,$ or if $k=1$ and $n$ is not a rhombic number),
then we say the cost is infinite.
Finally, we define the {\em minimum cost} $\cost(n)$ to be the minimum cost of
any equivalence relation of size $n$. Specifically,
\begin{align*}
\cost(n) &=\min \{\cost(n,k)\,:\, 1\le k \le n\}\\
&=\min \{\cost(n,k)\,:\, 1\le k \le n/2\}.
\end{align*}
The following result makes the relationship between $V(f)$ and total cost
explicit.
\begin{thm}\label{costProp}
Let $f$ be a function on $A$ with $N_2(f)=n$.
Define an equivalence relation on $A$ by $x\sim_f x'$ if and only if $f(x)=f(x').$ We have
$$V(f) = q - \cost(\sim_f) \le q - \cost(n).$$
\end{thm}
\begin{proof}
Note that an equivalence class of $\sim_f$ consists precisely of all elements
with the same image under $f$.
For each equivalence class of size $k,$ there are exactly $k(k-1)$ ordered pairs of distinct elements chosen from the equivalence class.
The contribution of this one equivalence class to the 
total cost of $\sim_f$ is $k-1$. For any set of $k$ elements of the domain that are mapped to the same element in the image of $f,$ we are losing $k-1$ potentially distinct elements of the image set. Now, $\cost(\sim_f)$ gives an exact
count of the number of lost images.
The upper bound now follows immediately from the definition of minimum cost and
the observation that $||\sim_f||=N_2(f)$.
\end{proof}
In light of this result, determining an upper bound for $V(f)$ with specified
$N_2$ is equivalent
to determining minimum cost. In particular, we have the following corollary.
\begin{cor}\label{Bkisdetermined}
With $B_k$ and minimum cost as defined above, $B_k=\cost(2k)$.
\end{cor}
We end this section by recording a couple of useful identities regarding the
rhombic floor.
\begin{lem}\label{sBounds}
Let $n$ be an even natural number. Set $t=\rhombicfl{n}$ and
$d=n-t(t-1)$. Then we have
$$\sqrt{n}-1 < t \le \sqrt{n}+1
\text{ and }
0 \le d \le 2t-2 < 2\sqrt{n}.$$
\end{lem}
\begin{proof}
Firstly, if $t\le\sqrt{n}-1$, then 
$$n<(t+1)t \le \sqrt{n} \left(\sqrt{n}-1\right)  <n,$$
contradicting $t=\rhombicfl{n}$.
Similarly, if $t\ge\sqrt{n}+1$, then
$$t(t-1) \ge \left(\sqrt{n}+1\right)\sqrt{n} > n,$$
again contradicting $t=\rhombicfl{n}$.

If $d\ge 2t$, then we find
$$n=t(t-1)+d \ge t(t-1)+2t = (t+1)t,$$
once again contradicting $t=\rhombicfl{n}$. Since $t(t-1)$ and $n$ are even, we
see $d\le 2t-2$, and the subsequent bound on $d$ involving $\sqrt{n}$ now 
follows at once.
\end{proof}
Obviously, we could get slightly better bounds by appealing to the quadratic
formula, but we will use the stated bounds in the lemma for simplicity.
This choice does not ultimately impact our result as we will eventually obtain
an exact solution.

\section{An algorithmic determination of $\cost(n)$} \label{CostSection}

Based on the previous corollary, to improve the upper bound of (\ref{bounds})
for $f\in C_4(A)$, we need to determine a good estimate of the minimum cost
$\cost(n)$ for $n=q-1$.
Through the course of this section we will, in fact, provide an algorithm for
determining $\cost(n)$ for any specified $n$.
Indeed, we shall eventually prove that, with some
small number of exceptions, a greedy algorithm approach will produce
$\cost(n)$ exactly.

We start by recording some simple lemmata that give us upper and lower bounds
on these cost minima. The first is an upper bound on the minimum cost of any equivalence relation of a given size. 

\begin{lem}\label{costUB}
For an even natural number $n$, we have
\[\emph{cost}(n)\leq\frac{-3+\sqrt{9+12n}}{2}.\]
\end{lem}
\begin{proof}
It is straightforward to check directly that the claimed bound holds for $n=2$ or $n=4.$ Henceforth, we assume $n\geq 6.$ We will first prove a bound for the cost of any equivalence relation of a given size with a prescribed number of classes, then use this to derive the claimed bound. Suppose $\sim_k$ is an equivalence relation with size $n$ that has exactly $k$ non-singleton equivalence classes, whose sizes are $s_j$, where $j=1, \dots, k.$ This gives us
\[n=\sum_{j=1}^k (s_j^2-s_j)\text{ and } \cost(\sim_k)=\sum_{j=1}^k(s_j-1)=\left(\sum_{j=1}^k s_j\right)-k.\]
Combining the above observations with Cauchy-Schwarz, we have
\[(\cost(\sim_k) + k)^2= \left(\sum_{j=1}^k s_j\cdot 1 \right)^2\leq \left(\sum_{j=1}^k s_j^2\right)\left(\sum_{j=1}^k 1^2\right)=\left(\sum_{j=1}^k s_j^2\right)k.\]
We now add and subtract $(s_j-1)$ in each term, and separate the sum to get
\[\left(\sum_{j=1}^k s_j^2-(s_j-1)+(s_j-1)\right)k=\left[\left(\sum_{j=1}^k (s_j^2-s_j)\right)+\left(\sum_{j=1}^k s_j-1\right)+\left(\sum_{j=1}^k 1\right)\right]k=\left[n + \cost(\sim_k)+k\right]k.\]
Putting this all together yields
\[(\cost(\sim_k)+k)^2\leq\left[n + \cost(\sim_k)+k\right]k\]
\[\cost^2(\sim_k)+2k\cost(\sim_k)+k^2\leq kn + k\cost(\sim_k)+k^2\]
\[\cost^2(\sim_k)+k\cost(\sim_k)\leq kn.\]
Solving for the largest possible value of $\cost(\sim_k)$ using the quadratic formula, we get
\begin{equation}\label{rawUB}
\cost(\sim_k)\leq\frac{-k+\sqrt{k^2+4kn}}{2}.
\end{equation}
Consider an equivalence relation $\approx$ of size $n$ that has three or fewer non-singleton equivalence classes with $\cost(\approx)\leq\cost(n,3).$ We are guaranteed that such an $\approx$ must exist because Gauss' well-known result that any natural number can be written as a sum of at most three triangular numbers implies that any even number can be written as a sum of at most three rhombic numbers. By appealing to \eqref{rawUB} with $k=3$, we see that
\begin{equation}\label{3UB}
\cost(\approx)\leq\frac{-3+\sqrt{9+12n}}{2}.
\end{equation}

Finally, let $\sim$ be an equivalence relation with size $n$ and realizing the minimal possible cost, so that \[\cost(\sim) = \cost(n).\]
Now, if $\sim$ has three or fewer equivalence classes, then by \eqref{3UB}, we see
\[\cost(n) = \cost(\sim)\leq\frac{-3+\sqrt{9+12n}}{2}.\]
If $\sim$ has more than three equivalence classes, then by definition, it must satisfy $\cost(\sim)\leq\cost(\approx),$ and the claim is proved by appealing to \eqref{3UB}.
\end{proof}

The next lemma gives a lower bound on the cost of any equivalence relation given some information about how many non-singleton equivalence classes it has using some basic convex geometry. The idea is that for any equivalence relation of size $n$, with $k$ non-singleton equivalence classes, the sizes of the equivalence classes must live on some $k$-dimensional sphere. Then we notice that the level sets of the contributions to the cost of any such equivalence relation will be constant on a family of hyperplanes.

\begin{lem}\label{costLB}
For natural numbers $2\leq k<n$, we have $\cost(n,k)>\sqrt{n-2k+\frac{9}{4}}+k-\frac{3}{2}.$ Moreover, if $\cost(n,1)$ is finite, we have $\cost(n,k) > \cost(n,1)>\sqrt{n}-1,$ and otherwise, we have $\cost(n,k) > \sqrt{n}-1.$ Finally, $\cost(n)>\sqrt{n}-1.$
\end{lem}
\begin{proof}
We need only consider the case where the cost is finite. In the case that $k=1,$ there must be a natural number $t$ so that $n=t(t-1).$ Moreover, the cost of this equivalence relation will be exactly $t-1.$ Now, we see that $\sqrt n<t,$ so $\sqrt{n} - 1<t-1,$ which is the cost of the equivalence relation, and the claim is proved for the case $k=1$.

In the case that $k=2,$ we must have natural numbers $x$ and $y$ so that $n=x(x-1)+y(y-1),$ where there is an equivalence class of exactly $x$ elements and one of exactly $y$ elements. Arguing as above, the total cost is $f(x,y) = (x-1)+(y-1).$ So for every equivalence relation of size $n$ with exactly two non-singleton equivalence classes of sizes $x$ and $y$, we have a pair $(x,y)$ satisfying the equation
\[x^2-x+y^2-y=n.\]
By completing the square, we see that the solution set corresponds to points on a circle given by the equation
\[\left(x-\frac{1}{2}\right)^2+\left(y-\frac{1}{2}\right)^2 = n+\frac{1}{2},\]
which is a circle of radius $\sqrt{n+(1/2)},$ centered at $\left(\frac{1}{2},\frac{1}{2}\right).$ Call this circle $C_n.$ Now we notice that any relevant equivalence relation with cost $\alpha$ will satisfy $\alpha = f(x,y) = x+y-2.$ For a fixed value of $\alpha,$ this gives the equation of a line, which we call $L_\alpha.$ Without loss of generality, we can assume $x\leq y.$ Since there are exactly two non-singleton equivalence classes, we know that both $x$ and $y$ are at least 2. Putting this all together, we see that the pair $(x_0,y_0)$ where $x_0=2$ and $y_0$ satisfies $n=2(1)+y_0(y_0-1)$ will give a point of intersection of the relevant circle and line. Solving this for $y_0$ we get $y_0=\sqrt{n-(7/4)}+(1/2).$ Note that this value of $y_0$ may not be a natural number, meaning that this is not a possible size for an equivalence class. We will see that this still provides a lower bound for $\cost(n,2).$ Define $\alpha_0$ by computing
\[f(x_0,y_0) = x_0+y_0-2 = 2+\left(\sqrt{n-\frac{7}{4}}+\frac{1}{2}\right) - 2 = \sqrt{n-\frac{7}{4}}+\frac{1}{2}=\alpha_0.\]
Since $f$ is symmetric in $x$ and $y$, we see that $f(x_0,y_0)=f(y_0,x_0)=\alpha_0.$ So we have the coordinates for both points where the line $L_{\alpha_0}$ could possibly intersect the circle $C_n.$ Notice that any line of the form $f(x,y)=\alpha$ must have slope $-1$. This means any such line is perpendicular to the line containing the origin and $\left(\frac{1}{2},\frac{1}{2}\right),$ the center of the circle $C_n.$ From this we see that for any $\beta<\alpha_0,$ the corresponding line $L_\beta$ would intersect the circle $C_n$ in a point with a coordinate strictly less than two, the coordinates of that point could not correspond with the sizes of non-singleton equivalence classes. Similarly, any other line with slope $-1$ would intersect the circle $C_n$ in points whose coordinates $(x,y)$ are both strictly greater than two, but would then give a value of $f(x,y)$ strictly greater than $\alpha_0.$ So we conclude that $\alpha_0$ is a lower bound for $\cost(n,2)$ (see diagram).
\begin{figure}
\includegraphics[scale=1.25]{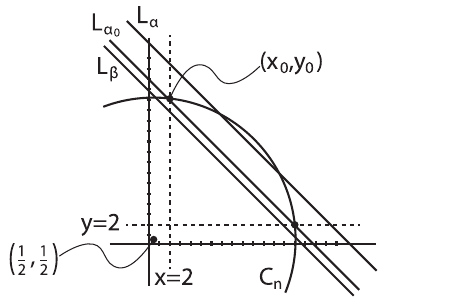}
\caption{Here we see the circle $C_n$ centered at $\left(\frac{1}{2},\frac{1}{2}\right)$, and $L_{\alpha_0}$ intersecting $C_n$ at $(x_0,y_0)$ and $(y_0,x_0).$ Also pictured are examples of lines $L_\alpha$ with $\alpha>\alpha_0,$ and $L_\beta$ with $\beta<\alpha_0.$ }
\end{figure}

For larger values of $k$, recall that $k$ is finite, so we can appeal to higher, yet still finite-dimensional arguments. Suppose $\vec x = (x_1, \dots, x_k),$ where the $x_j$ are the sizes of the $k$ different non-singleton equivalence classes. The argument is essentially the same as the $k=2$ case, except we now have a sphere in place of a circle, whose center has coordinates all $\frac{1}{2},$ and hyperplanes in place of lines, all normal to the line through the origin and the center of the sphere. The equation of the sphere will be
\[\left(x_1-\frac{1}{2}\right)^2+\left(x_2-\frac{1}{2}\right)^2+\dots+\left(x_k-\frac{1}{2}\right)^2 = n+\frac{k}{4},\]
and the relevant hyperplanes will be defined by
\[f(x_1,x_2,\dots,x_k) = (x_1-1)+(x_2-1)+\dots+(x_k-1) = \alpha.\]
The minimum over the reals would then occur whenever $(k-1)$ of the coordinates are exactly $2$ (the minimum possible value for any $x_j$), and the last coordinate is $n-2(k-1).$ Following the higher-dimensional analog through, we find the hyperplane going through the point with coordinates
\[\left(2,2, \dots, 2, \sqrt{n-2k+\frac{9}{4}}+\frac{1}{2}\right),\]
and call it $H_{\alpha_0}.$ As before, we compute
\[\alpha_0 =  2(k-1)+\left(\sqrt{n-2k+\frac{9}{4}}+\frac{1}{2}\right)-1 = \sqrt{n-2k+\frac{9}{4}}+2(k-1)-\frac{1}{2}.\]
Again, we see that parallel hyperplanes closer to the origin than $H_{\alpha_0}$ will have some coordinates smaller than two, and parallel hyperplanes farther from the origin will correspond with a cost larger than $\alpha_0.$ So we get that $\alpha_0$ is a lower bound for $\cost(n,k).$

Finally, it is routine to check that in the range $n\geq 2k > 0$ that the lower bound for $\cost(n,k)$ is greater than the lower bound for $\cost(n,1)$ and greater than $\sqrt{n}-1.$ To see this, notice that if $n$ can be written as $t(t-1)$, then for any $k>1$, we will have 
\[\sqrt{n-2k+\frac{9}{4}}+k-\frac{3}{2} > t-1 > \sqrt{n}-1.\]
Notice that the left-hand side will still be strictly greater than the right-hand side even if the middle term is not present. Since this implies that $\cost(n,k)>\sqrt{n}-1$ for all $k$, $\cost(n)$ is just the minimum over all $k,$ and there are only finitely many choices for $k,$ we also get that $\cost(n)>\sqrt{n}-1.$
\end{proof}
We note that, in light of \Cref{costProp}, the upper bound of
\eqref{bounds}
also gives the slightly better lower bound of
$\cost(n)\ge \frac{1}{2}\left(-1+\sqrt{4n+1} \right)$, but we will use the lower bound of the
lemma for simplicity in future calculations.

One immediate implication of \Cref{costLB} is the following.
\begin{cor}\label{nRhombic}
If $n=t(t-1),$ for a natural number $t,$ then $\cost(n)=t-1,$ and any equivalence relation realizing this bound must have exactly one non-singleton equivalence class of size $t.$
\end{cor}
\begin{proof}
We have $\cost(n,1)=t-1$
as $n=t(t-1)$. In particular, $\cost(n,1)$ is finite, so that \Cref{costLB}
implies $\cost(n,k)>\cost(n,1)$ for all $k>1$. Thus, $\cost(n)=t-1$ and this can
only be achieved if there is exactly one non-singleton
equivalence class of size $t$.
\end{proof}
We now prove our main result for this section, perhaps the main result of the
paper.
\begin{thm}\label{nonTriCost}
Fix $n\ge 2$ be even and let $\sim$ be an equivalence relation on a finite set
$A$ such that
\begin{itemize}
\item $|| \sim ||=n$.
\item $\cost(\sim)=\cost(n)$, so that $\sim$ has minimal total cost among
all such equivalence relations.
\end{itemize}
Then $\sim$ must have an equivalence class of size $\rhombicfl{n}$, unless
\begin{equation*}
n\in\X:=\left\{
\begin{array}{cccccccc}
24,
&40,
&50,
&52,
&86,
&88,
&120,
&128,\\
174,
&180,
&198,
&238,
&266,
&268,
&296,
&300,\\
378,
&414,
&534,
&690,
&740,
&866,
&922,
&980,\\
982,
&1188,
&1254,
&1256,
&1692,
&1962,
&2136,
&2344,\\
2438,
&2440,
&2632,
&2966,
&2968,
&3180,
&3780,
&3870,\\
4530,
&4662,
&5384,
&5528,
&6312,
&6468,
&10280,
&11550,\\
11762,
&20574,
&22950&&&&&
\end{array}
\right\}.
\end{equation*}
\end{thm}

We pause to note that the restriction of dealing with finite underlying sets is entirely superficial, as all of the relevant quantities we manipulate are finite, however rather than teasing out this technicality, we add the assumption of finiteness for ease of exposition, as well as because the primary application of this result in the present paper is dealing with finite sets. We also point out that this exceptional set $\mathcal X$ is indeed necessary. That is, our methods yield exact results, and $\mathcal X$ is a complete list of cases where the greedy approach does not uniquely determine the minimal cost.

\begin{proof}
Assume for now that $n\ge 2574420=1605\times 1604$. Our goal is to show that
the minimal total cost must arise from an equivalence relation with an
equivalence class of the maximal possible size.
Set $t=\rhombicfl{n}$ and let $d=n-t(t-1)$.
We have $0\leq d \leq 2t-2$ by \Cref{sBounds}.
Moreover, by our assumption on the size of $n$, we know that $t\geq 1605.$
Notice that in the case that $d=0,$ we are done by \Cref{nRhombic}. 
We need to deal with the case where $d> 0.$

The basic plan will be to show that if an arbitrary equivalence relation $\equiv$ of size $n$ does not have an equivalence class of size $t,$ then it must have a cost strictly greater than $\cost(n).$ The difficulty lies in the fact that, as already noted, this does not always hold for small $n$. To find a sufficient threshold, we write $\cost(n)$ in terms of $\cost(d),$ and relate this to the cost of $\equiv.$ We then assume the conclusion of the result fails, and this will give us an upper bound on $d.$ We pause here to note that an upper bound on $d$ on its own would not be enough to prove the claim for large $n$, as there are arbitrarily large values of $n$ whose corresponding $d$ will be less than any positive bound. However, due to the relationship between $\cost(d)$ and the cost of $\equiv,$ we can eventually show that for sufficiently large $n$, the conclusion must hold. In some sense, we are separating out the long-term growth of the cost function from its short-term fluctuations.

Define $\mathcal E(n)$ to be the set of equivalence relations on $A$ of size $n.$
Further, for an arbitrary integer $1\le j\le t$, let $\mathcal E_j(n)$ denote the
equivalence classes from $\mathcal E(n)$ having a largest equivalence class of
size $j$.
Choose $\sim$ to be an equivalence relation chosen from $\mathcal E_t(n)$ to
have minimal cost out of all equivalence relations in $\mathcal E_t(n),$ which
can be done as all of the relevant sets involved are finite. As $\sim$ has an
equivalence class of size $t,$ the contribution to $\cost(\sim)$ by that
class is $t-1$. Since $n>4,$ by the definitions of $t$ and $d,$ there is unique equivalence class of size $t$ in $\sim.$ Therefore each of the remaining $d$ related ordered pairs of
distinct elements must be in smaller equivalence classes, and therefore have a combined cost at most $\cost(d)$. Since $\sim$ was chosen to have minimal cost out of all equivalence relations in $\mathcal E_t(n),$ this shows
\begin{equation}\label{simCost}
\cost(\sim)=(t-1)+\cost(d).
\end{equation}
Next, for any natural number $k,$ if possible, select an arbitrary $\approx_k\in\mathcal E_{t-k}(n).$ That is, $\approx_k$ is an equivalence relation whose largest non-singleton equivalence class has size $t-k$ for some $k \geq 1,$ meaning that there is a difference of $k$ elements between the largest possible size of an equivalence class and the largest equivalence class in this particular relation. Note that for $\approx_k$ to exist, this implies that
\[t-k \geq 2 \Rightarrow k \leq t-2,\]
because $t-k<2$ would imply that this largest equivalence class was a singleton. However, if we are considering an equivalence relation that could potentially have minimal cost, we can get more control on $k.$ Specifically, we will initially look for a threshold value, $k_0$, so that for all $k\geq k_0,$ the conclusion holds by a fairly simple argument, then use a different argument for smaller choices of $k.$ To this end, start with some $k>0.$ Each equivalence of distinct elements $x\approx_k x'$ must come from an equivalence class of size at most $t-k,$ and therefore contribute $w(x,x')\geq(t-k)^{-1}$ to the cost. So we have
\[\cost(\approx_k)\geq n\cdot\left(\frac{1}{t-k}\right) = \frac{n}{t-k}.\]
If $\cost(\approx_k)>\cost(\sim),$ then we are already done for that value of $k$ and larger. That is, for large enough $k$, the conclusion of the theorem holds without arguing further. We check to see which values of $k$ will already push $\approx_k$ to have too high a cost with this rough estimate. By appealing to \eqref{simCost} and \Cref{costLB}, we get
\[\cost(\approx_k)\geq \frac{n}{t-k}>\cost(\sim) = (t-1)+\cost(d) > (t-1) + \sqrt{d} -1.\]
Notice this holds whenever
\[n>\left(t+\sqrt{d}-2\right)(t-k)=t^2 +t\sqrt{d}-2t -k\left(t +\sqrt{d}-2\right).\]
Recalling $n=t(t-1)+d,$ we have
\[(t^2-t+d)-t^2-t\sqrt{d}+2t > k\left(2-t-\sqrt{d}\right).\]
Simplifying and multiplying by $-1$ yields
\[t\sqrt{d}-t-d < k \left(t+\sqrt{d}-2\right),\]
which we use to define the threshold value $k_0,$ by
\[k> \frac{t\sqrt{d}-t-d}{t+\sqrt{d}-2}=:k_0.\]
While this is a rather intricate expression, it implies that for any $k>\sqrt{d}$
\begin{equation}\label{kSmall}
k> \sqrt{d}>k_0.
\end{equation}
Now that we have some idea about the range of $k$ where we need to work, we follow the logic that lead to \eqref{simCost} above to see that any equivalence class with size $(t-k)$ will contribute $(t-k-1)$ to the cost of $\approx_k.$ Recall that we want to compare the cost of the equivalence relation $\approx_k$ to that of $\sim,$ so now that we have accounted for the cost contribution by a maximal equivalence class in $\approx_k,$ we next need to count how many remaining equivalences there are to estimate the rest of its cost, in addition the $d$ equivalences already left over in $\sim.$ We do this by noting
\[t(t-1) - (t-k)(t-k-1) = (t^2-t)-(t^2-2tk+k^2-t+k)=2kt-k^2-k.\]
Recalling that there are still $d$ other equivalences that are as yet unaccounted for, this gives us the following estimate for the cost of $\approx_k.$
\[\cost(\approx_k)\geq (t-k-1) +\cost(d+ 2kt-k(k+1)).\]
We want to show that $\cost(\approx_k)> \cost(\sim).$ By combining the estimate above with \eqref{simCost}, this gives us the goal of showing
\[(t-1)+\cost(d) < (t-k-1) +\cost(d+2kt-(k+1)k)\]
\begin{equation}\label{kGoal}
\cost(d) < \cost(d+2kt-k(k+1))-k.
\end{equation}
Notice that \eqref{kGoal} does not always hold in general, as there are exceptions for small values of $n$. It should not be surprising that it holds for large values of $n$, as the cost of $d$ on the left-hand side should stay relatively small compared to the cost term on the right-hand side as the parameter $t$ grows. However, there is a confounding loss term of $k$ on the right-hand side as well. Since all of the parameters involved can grow in $n$, a direct approach does not appear to work. Instead, we will tease out that this indeed holds as $n$ grows sufficiently large by handling different ranges of $k$ with slightly different arguments. In order to guarantee \eqref{kGoal} holds for $n$ sufficiently large, we will assume it fails, and for a given $k$, get conditions on $d$ or $t$ necessary (though potentially not sufficient) for \eqref{kGoal} to fail. In particular, suppose that the opposite holds. That is, suppose
\begin{equation}\label{kFail}
\cost(d) \geq \cost(d+2kt-k(k+1))-k.
\end{equation}
We will break $d$ down further by defining $s=\rhombicfl{d}$, and letting $d'$ satisfy $d=s(s-1)+d'.$
By \Cref{sBounds}, we have $0\leq d' \leq 2s-2$ and
$$\sqrt{d}-1 < s < \sqrt{d}+1.$$
Following the reasoning above, and appealing to \Cref{nRhombic}, we
estimate
\[ \cost(d) \leq \cost(s(s-1))+\cost(d') = s-1 + \cost(d').\]
Combining this with \eqref{kFail}, and applying \Cref{costLB}, we get
\[s-1 +\cost(d') \geq \cost(d) \geq \cost(d+2kt-k(k+1))-k > \sqrt{d+2kt-k(k+1)}-1 -k,\]
which simplifies to
\[\cost(d') > \sqrt{d+2kt-k(k+1)} - k - s.\]
We then apply \Cref{costUB} to get
\[\frac{-3+\sqrt{9+12d'}}{2} \geq \cost(d') > \sqrt{d+2kt-k(k+1)} - k - s,\]
which simplifies to
\begin{equation}\label{splitBound}
\sqrt{9+12d'} > 2\sqrt{d+2kt-k(k+1)} - 2k - 2s+3.
\end{equation}

We are looking for a range of parameters that could potentially satisfy \eqref{splitBound}. The upper bound on $s$ from \Cref{sBounds} yields
\[\sqrt{9+12d'} > 2\sqrt{d+2kt-k(k+1)} - 2k - 2s+3 > 2\sqrt{d+2kt-k^2-k} - 2k - 2\left(\sqrt{d}+1\right)+3.\]
Since $d\leq2t-2$, this is bounded below by
\[\sqrt{9+12d'} > 2\sqrt{d+(2kt-2k)-k^2+k} - 2k - 2\sqrt{d}+1\geq 2\sqrt{d+kd-k(k-1)} - 2k - 2\sqrt{d}+1.\]
We then note that $\sqrt{d}\geq\sqrt{d-(k-1)}$, to get
\[\sqrt{9+12d'} > 2\sqrt{d+kd-k(k-1)} - 2k - 2\sqrt{d-(k-1)}+1.\]
Again using \Cref{sBounds}, we have $d' < 2\sqrt{d}$, so that
\[\sqrt{9+24\sqrt{d}} > \sqrt{9+12d'} >2\sqrt{d+kd-k(k-1)} - 2k - 2\sqrt{d-(k-1)}+1\]
\begin{equation}\label{kSplit}
\sqrt{9+24\sqrt{d}} > 2\sqrt{d+k(d-(k-1))} - 2k - 2\sqrt{d-(k-1)}+1.
\end{equation}
At this point we will split our problem into cases based on $k.$ Now, if $k\geq 9$, we ignore one of the $d$ terms under the first square root on the right-hand side, which we can do with a lower bound, and obtain
\[\sqrt{9+24\sqrt{d}} > 2\sqrt{k(d-(k-1))} - 2k - 2\sqrt{d-(k-1)}+1> 2\left(\sqrt{k}-1\right)\sqrt{d-(k-1)} - 2k+1.\]
Applications of the fact that $k\geq 9$ and \eqref{kSmall} give us
\[\sqrt{9+24\sqrt{d}}>4\sqrt{d-(k-1)} - 2k+1>4\sqrt{d-\sqrt{d}+1}-2\sqrt{d} +1.\]
Solving with a computer will yield that $d<49.9822\dots$, which implies $d\leq 48$ as $t$ must be an even integer. We can explicitly calculate the maximum possible cost of any equivalence relation with $d\leq 48$,  and plug these into \eqref{kGoal} to get bounds on $t.$ Checking by hand, we see that for $d\leq 48,$ we always have $\cost(d)\leq 8,$ so we look for bounds on $t$ satisfying
\begin{equation}\label{dPrimeLarge1}
8 \leq \cost(d) < \cost(d+2kt-k(k+1))-k.
\end{equation}
Appealing to \eqref{kSmall}, and recalling the range of $k$ we are currently analyzing, we can bound $9\leq k\leq \sqrt{d},$ so the argument of the cost function in \eqref{dPrimeLarge1} is
\[d+2kt-k(k+1) \geq d+2(9)t -\sqrt{d}\left(\sqrt{d}+1\right) \geq 18t-\sqrt{d}>18t.\]
So we apply \Cref{costLB} to the right-hand side of \eqref{dPrimeLarge1} and combine with the above calculation to get
\[\cost(d+2kt-k(k+1))-k \geq \sqrt{d+2kt-k(k+1)}-1-k >\sqrt{18t}-k-1=3\sqrt{2t}-k-1.\]
We lean on \eqref{kSmall} again to see that $k\leq \sqrt{d} \leq \sqrt{2t}.$ Therefore, we can be assured that \eqref{kGoal} holds whenever \eqref{dPrimeLarge1} does, which is when
\[\cost(d+2kt-k(k+1))-k> 3\sqrt{2t}-\sqrt{2t}-1 = 2\sqrt{2t} -1\geq 8.\]
This happens whenever $t\geq10.125,$ which we are guaranteed by the assumption on the size of $n.$

We now consider \eqref{kSplit} when $1\leq k \leq 8$:
\[\sqrt{9+24\sqrt{d}} > 2\sqrt{d+k(d-(k-1))} - 2k - 2\sqrt{d-(k-1)}+1.\]
For this range of $k$, we check each case separately by computer and see that for each value of $k$, there is a maximum value of $d$ for which the inequality can hold. Out of each of these values of $k$, the maximum possible value for $d$ to satisfy this inequality occurs when $k=1,$ which has a maximum for $d$ at no more than $1421.$ 
By again following the same general strategy used to analyze \eqref{dPrimeLarge1}, but with these possible values for $k$, we see that we need to compute the minimum cost of any equivalence relation of size $d<1421.$ Since $d$ must be even, we get that $d\leq 1420.$
A simple recursive algorithm can be set up to compute all values of $\cost(d)$
for relatively small $d$, and this shows that $\cost(d)\le 47$ for all
$d\le 1420$, with equality occurring only for $d\in\{1398,1402\}$.
This means that we can realize \eqref{kGoal} as long as the following
inequality holds, which will serve a purpose similar to that of
\eqref{dPrimeLarge1} above:
\begin{equation}\label{dPrimeLarge2}
47 < \cost(d+2kt-k(k+1))-k.
\end{equation}
Again recalling the range of $k$ in the current case, $1\leq k\leq 8,$ we see the argument of the cost function in \eqref{dPrimeLarge2} is
\[d+2kt-k(k+1) \geq d+2t -(8)(9) \geq 2t-72.\]
Similar to before, we apply \Cref{costLB} to the right-hand side of \eqref{dPrimeLarge2} and once again combine with the above estimate and range on $k$ in this case to obtain
\[\cost(d+2kt-k(k+1))-k \geq \sqrt{d+2kt-k(k+1)}-1-k >\sqrt{2t-72}-9.\]
Again, we can be assured that \eqref{kGoal} holds whenever \eqref{dPrimeLarge2} does, which is when
\[\cost(d+2kt-k(k+1))-k> \sqrt{2t-72}-9 \geq 47.\]
This happens whenever $t>1604,$ which we are again guaranteed by our range on $n,$ which is at least $2574420=1605\times 1604$ by assumption.

One can now return to the same recursive algorithm mentioned above to compute
the minimal examples for all $n < 2574420$ easily enough, and
confirm the statement for all $n$ not listed among the 51 exceptions of $\X$.
The exceptions fall into two categories. If $n\in \{40,238,922,1188,2344\}$,
there exists an equivalence relation with minimal total cost that is 1 less
than that which can be achieved from having an equivalence class of size
$\rhombicfl{n}$.
For the remaining exceptions, there are equivalence relations with minimal
total cost equal to that which can be achieved from having an equivalence
class of size $\rhombicfl{n}$.
\end{proof}
We now apply \Cref{nonTriCost} to get the following result, which describes the minimal possible cost for any equivalence relation on a finite set.
The point is that, apart from the small number of exceptions, greedily using up
as many equivalences as we can at any given step in the process will yield the
minimal total cost.
Indeed, once an equivalence class of maximal size is fixed, we can restrict our
consideration to the sub-equivalence relation (in the Cartesian product sense)
and reapply \Cref{nonTriCost}.
Modifying the larger classes will only increase the cost, so we can move on to smaller equivalence classes.
\begin{thm}\label{cobbleCost}
Let $A$ be a finite set and $\sim$ be an equivalence relation on $A$ with
$|| \sim ||=n$.
Set $e_1=\rhombicfl{n}$, and recursively define the finite sequence of
integers $\{e_j\}$ by
\[e_j = \rhombicfl{ n - \sum_{i<j} e_i (e_i-1)}.\]
If all of the $e_j$ are not among the set $\mathcal X' :=\{40, 238, 922, 1188, 2344 \}$, then
$\cost(\sim)\ge \sum_j (e_j-1)$, and in particular,
$\cost(n) = \sum_j (e_j-1)$. If any one of the $e_j$ is in $\mathcal X',$ then it will be the only such exception, and $\cost(n) = -1+\sum_j (e_j-1).$
\end{thm}
\begin{proof}
Given a fixed $n$, \Cref{nonTriCost} guarantees that if $\sim$ has
minimal total cost, then by hypothesis the largest equivalence class must have
size $e_1=\rhombicfl{n}$.
This will account for $e_1(e_1-1)$ equivalences of ordered pairs of distinct
elements, altogether contributing $e_1-1$ to the total cost.
This leaves us with $n_1 = n-e_1(e_1-1)$ remaining equivalences of ordered
pairs of distinct elements. Denote the equivalence class of size $e_1$ by
$E_1$. We can again appeal to \Cref{nonTriCost}
applied to any equivalence relation on the set $A\setminus E_1$ to get that the
remaining portion of the total cost is minimized when the next largest
equivalence class has size $e_2 = \rhombicfl{n_1}$.
Since our hypothesis guarantees that no $e_j$ corresponds to an exception,
we can recursively generate the rest of the $e_j$ by repeatedly applying
\Cref{nonTriCost} to the remaining portions of $A$.
As the set $A$ is finite, this process will eventually terminate.
\end{proof}
In general, this sum is minimized when all of the equivalent pairs are between elements from the same equivalence class. This agrees with the putative sharpness example where $q=343$.
To summarize, to determine $B_k$, or an upper bound for $V(f)$ with
$f\in C_4(A)$, a greedy algorithm approach will almost always produce the
largest image size, so long as we don't encounter one of the exceptions.
While this seems intuitively obvious, the exceptions make this somewhat
involved.
Morover, while the problem might appear to be within the reach of standard
discrepancy techniques, the finite but sparse list of exceptions seem to
confound such approaches. One reason this occurs is that
that the main term and the error term can both be of order $\sqrt{q}$ in
magnitude.

For completeness, we have given some additional information regarding the
exceptions in an appendix, including their exact costs.
We summarize the more important aspects:
\begin{itemize}
\item The only cases where $\cost(n)$ cannot be generated by an
equivalence relation with an equivalence class of size $\rhombicfl{n}$ are
$n\in\{40,238,922,1188,2344\}$. In each of these cases, $\cost(n)$ is one less
than the minimum cost of an equivalence relation with a maximal size equivalence
class, and is achieved with largest equivalence class of size
$\rhombicfl{n}-1$.

\item If $n\not\in\{40,238,922,1188,2344\}$, then $\cost(n)$ can be achieved by
an equivalence relation that has an equivalence class of the maximal size
$\rhombicfl{n}$.

\item For $n=128$, we have $\rhombicfl{128}=11$ and $\cost(128)=15$.
This cost can be achieved by an equivalence relation with just 2 non-singleton
equivalence classes of sizes 9 and 8, respectively.

\item For $n=180$, we have $\rhombicfl{180}=13$ and $\cost(180)=18$.
This cost can be achieved by an equivalence relation with just 2 non-singleton
equivalence classes each of size 10.

\item For $n=300$, we have $\rhombicfl{300}=17$ and $\cost(300)=23$.
This cost can be achieved by an equivalence relation with just 2 non-singleton
equivalence classes of sizes 15 and 10, respectively.

\item In all other exceptions, $\cost(n)$ can be achieved only by an equivalence
relation having a largest equivalence class of size 
$\rhombicfl{n}$ or of size $\rhombicfl{n}-1$.

\item There are 19 cases where $q$ is a prime and $q-1\in\X$. These
are
\begin{equation*}
q\in \left\{
\begin{array}{c}
41, 53, 89, 181, 199, 239, 269, 379, 691, 983, 1693,\\
2137, 2441, 2633, 2969, 3181, 4663, 6469, 11551
\end{array}
\right\}
\end{equation*}

\item There are 2 cases where $q$ is a non-trivial power with
$q-1\in\X$; namely, $q=25$ and $q=121$. In particular, if $q=y^n$ for some
$n\ge 3$, then $q-1\notin\X$.
\end{itemize}

\section{The general situation} \label{gensit}

We now want to give bounds on $V(f)$ with $f\in C_4(A)$ for general $q$.
\Cref{cobbleCost} solves the problem algorithmically for any specified $q$, but it still does not communicate a general sense of the limits on $V(f)$.
The exceptions of \Cref{nonTriCost} are effectively
dealt with via the Appendix and \Cref{costProp}, so we set those
aside in all that follows, as we do the case $q=343$.

We start by establishing a new upper bound which, though not a great
improvement on \eqref{bounds}, will serve as the example for the approach
used in the remainder of this paper.
\begin{lem} \label{firstnewbound}
Let $q=y^n$ with $y>1$, $q\neq 343$ and $q-1\notin\X$.
If $n>1$ or there exists a prime $p\not\equiv 1\bmod 6$ dividing $q$,
then any function $f\in C_4(A)$ satisfies
\begin{equation}\label{newbounds}
\frac{q+1}{2}\leq V(f) \leq q-\frac{1+\sqrt{4q-11}}{2}.
\end{equation}
\end{lem}
\begin{proof}
By \Cref{thm1} (ii), any function $f\in C_4(A),$ cannot attain the
upper bound in \eqref{bounds}.
By \Cref{nonTriCost}, there must be at least two equivalence classes
that are not singletons.
\Cref{cobbleCost} tells us we will minimize cost by maximizing the
largest equivalence class we are allowed.
To this end, we must therefore have a maximal large equivalence class, and one
more non-singleton equivalence class of size exactly two.
This can only happen if there is an integer $t$, so that $t(t-1)+2(2-1)=q-1.$
In this case, $q-3=t(t-1).$ Solving for $t$ gives us one real possibility,
\[t=\frac{1+\sqrt{4q-11}}{2}.\]
Then any such $f$ would have one element with $t$ pre-images, another element with $2$ pre-images, and be injective elsewhere. This would give
\[V(f)\leq q- (t-1) - (2-1)=q-t,\]
as claimed.
\end{proof}
With \Cref{firstnewbound} as a model, we now describe our general approach.
For the remainder, we assume $f\in C_4(A)$ and assume $q-1\notin\X$.
For $i\ge 0$, we set $D_i=8i+3$.

At Step 0 (which is the bound in \cite{coulter14c}), we had the bound
\begin{equation} \label{oldbound}
\frac{q+1}{2}\le V(f) \le  q - \left(\frac{\sqrt{4q-D_0} - 1}{2}\right).
\end{equation}
This was analyzed by first looking at $t_0(t_0-1)=q-1$, which leads to the
discriminant $4q-3=4q-D_0$ and hence the Diophantine equation $x^2+D_0=4q$.
Under the assumption $q=y^n$ with $y,n>1$, we find the sole solution $q=7^3$.

We now exclude the exception and proceed to Step 1.
The next possibility for the largest preimage set is $t_1$ coming from
$t_1(t_1-1)=q-3$. This produces the Diophantine equation
$x^2 + D_1 = 4y^n$.
This analysis and the application of \Cref{cobbleCost} leads to the
bound
\begin{equation*}
V(f) \le q - \left(\frac{\sqrt{4q-D_1} - 1}{2}\right) - 1.
\end{equation*}
This is the upper bound of \Cref{firstnewbound}, with the exception $q=7^3$.

Proceeding, possibly with (infinitely many) exceptions at each step,
we find that at Step $i$
we are considering first $t_i(t_i-1)=q-1-2i$ (remember, the $q-1$ part is
coming from $N_2(f)=q-1$, which produces the 
Diophantine equation $x^2+D_i  = 4y^n$), and the subsequent analysis via the
greedy algorithm produces the new bound
\begin{equation} 
V(f) \le q - \left(\frac{\sqrt{4q-D_i} - 1}{2}\right) - B_{2i},
\end{equation}
where the $B_{2i}$ is as defined at the start of \Cref{sec5}. Since $\rhombicfl{q-1}$ is well defined, there must be
an endpoint to this process.

\section{The case where $q$ is a square}

The following simple generalization of \Cref{reductions} is key to
establishing a new upper bound in the square case.
\begin{prop} \label{squaresoln}
Let $q=y^2$ with $y>1$ odd and $i\ge 0$.
If $2i+1 < y$, then $t(t-1)=q-1-2i$ has no integer solution.
Furthermore, when $2i+1=y$, we have the solution $t=y$.
\end{prop}
\begin{proof}
Firstly, note that if $2i+1=y$, then
$q-y = y(y-1)$, so that $t=y$ is a solution in that case.

Now suppose we have a solution 
$t(t-1)=y^2-1-2i$ with $2i+1<y$.
Subtracting $2i(2i+1)$ from both sides, we obtain
$$t^2-t-2i(2i+1) = (t+2i)(t-(2i+1)) = y^2 - (2i+1)^2.$$
Setting $a=t-(2i+1)$, $b=y-(2i+1)$ and $j=4i+1$, we get the integral equation
$$a(a+j)  = b(b+j+1),$$
with $b<a<a+j<b+j+1$. This implies there are two integers that differ by
$j$ that lie strictly between two integers that differ by $j+1$, which is
impossible.
\end{proof}
The result shows how we can prove \Cref{squarebound}.
\begin{thm2}
Fix a nonsquare integer $y>1$ and set $q=y^{2^j}$ for some $j\ge 1$.
Suppose $f\in C_4(A)$.
Then
\begin{equation*}
V(f) \le q + j - \sum_{i=0}^{j-1} y^{2^i} -
\begin{cases}
7 &\text{ if $y=5$,}\\
15 &\text{ if $y=11$, and}\\
\frac{-1+\sqrt{4y-3}}{2} &\text{ otherwise.}
\end{cases}
\end{equation*}
\end{thm2}
\begin{proof}
We exclude the cases $y=5$ and $y=11$ from our initial considerations as they
are the only situations where we can encounter an exception.
By \Cref{cobbleCost}, the largest possible image size for $f$ comes from
taking the largest preimage size we can at each step.
Since $N_2(f)=q-1$, we apply \Cref{squaresoln} with $n=q-1$.

If $q=y^{2^j}$, then \Cref{squaresoln} shows that the largest preimage size
is $t_1=y^{2^{j-1}}$.
This will leave a remainder of $t_1-1$ pairs to be accounted for
to ensure $N_2(f)=q-1$.
Another application of \Cref{cobbleCost} and \Cref{squaresoln} gives
our next largest preimage size is $t_2=y^{2^{j-2}}$ with a cost of $t_2-1$, and
leaving $t_2-1$ pairs to be accounted for. We now proceed recursively until we
get to $y-1$ remaining pairs. Since $y$ is itself not a square, we can no
longer appeal to \Cref{squaresoln}. The best we can say is that the
minimum possible cost to construct the remaining $y-1$ pairs will be 
$s-1$ where $s(s-1)=y-1$. Applying the quadratic formula now completes the
bound.

If we now consider the $y=5$ and $y=11$ cases, we see that the above argument
works essentially in the same way until we get to $y^2-1$ pairs, at which 
point we can appeal directly to the results of the Appendix. This yields the
bounds claimed.
\end{proof}
Note that if we knew $y=y_0^n$ for $y_0,n>1$ and $n$ odd, then we could
subsequently apply our previous knowledge to conclude this was a strict
inequality upper bound unless $y=343$. Indeed, the results of the following
sections, such as \Cref{newgeneralbound}, might also then be applied.

\section{More on the Lebesgue-Ramanujan-Nagell equation}

In preparation for improving the bound in the nonsquare case, we need to
revisit the general Diophantine problem that was at the heart of 
the argument of \Cref{step0section} and seen again in a slightly more general
form in \Cref{gensit}.

There is an immense amount of literature dealing with Diophantine problems
such as $x^2 + 3 = 4y^n$.
The result of Thue \cite{thu16} concerning (\ref{EQ1})
mentioned above is one of the more important general ones, but many more
results concerning variations of this equation have been produced
since (and before!) Thue's result.
Among a host of restricted settings, the equation
\begin{equation} \label{EQ2}
\alpha x^2 + \beta = \lambda y^n,
\end{equation}
with $\gcd(y,\alpha\beta)=1$ has received much attention.
Catalan \cite{cat44} famously conjectured in 1844 that the only non-trivial
consecutive perfect powers of natural numbers are 8 and 9. Put another way,
the only solution to the equation $x^m-y^n = 1$, with $x,y,m,n>1$, is given by
$(x,y,m,n)=(3,2,2,3)$.
Just 6 years later, Lebesgue \cite{leb50} showed the equation
$x^2+1=y^n$ had no integral solutions with $x,y,n>1$.
(Catalan's full conjecture was proved only in 2004
by Mih\u{a}ilescu \cite{mih04}.)
Much later, Nagell gave several in-depth treatments of the $\alpha=1$ case of
(\ref{EQ2}), see \cite{nag23, nag54}. Nagell also proved in \cite{nag48} a
conjecture of Ramanujan \cite{ram13} from 1913 that there were only five
triangular Mersenne numbers, or equivalently, that the only integer solutions
$(x,n)$ to $x^2+7=2^n$ were given by
$(x,n)\in\{(1,1),(3,2),(5,3),(11,5),(181,13)\}$.
Equations and generalizations of the form of Equation \ref{EQ2} are now often
called Lebesgue-Ramanujan-Nagell equations (or some variant) because
of this history.

As seen in \Cref{gensit}, equations of the form $x^2+D_i=4y^n$ are central to
our aims.
In recent times, the approach to resolving such equations has
made use of a landmark paper of Bilu, Hanrot and Voutier \cite{bilu01} which
resolved a long-standing problem on Lucas and Lehmer numbers with no primitive
divisors. 
Results of Bugeaud \cite{bugeaud01a}, with a correction pointed out by
Bilu in \cite{bilu02}, and of Arif and Al-Ali \cite{arif02} consider the more
general equation $\alpha x^2+\beta{2k+1} = 4y^n$ for arbitrary $y$ and $n\ge 5$ a prime.
Bugeaud and Shorey also prove a similar result in 
\cite{bugeaud01}, Corollary 7.
The case where $n=3$ is more complex.
However, Bugeaud and Shorey deal with the $n=3$ case as part of a full account
of the case where $y$ is a prime in \cite{bugeaud01}.
In another direction, 
Luca, Tengely and Togb\'e \cite{luca09} determined all solutions to the
equation $x^2+C=4y^n$ for $n\ge 3$, $C\equiv 3\bmod 4$ and $1\le C\le 100$,
as well as all solutions when $C=7^a11^b$ or $7^a13^b$.
The following lemma outlines the subset of those results relevant to our
problem.
We denote the class number of $\qf{D_i}$ by $h(-D_i)$.
\begin{lem} \label{eqlemma}
Let $D_i=8i+3$ for some $i\ge 0$.
Consider the equation
\begin{equation} \label{theeq}
x^2 +  D_i = 4y^n
\end{equation}
with $x>1$ arbitrary, $y,n\ge 3$ odd, and $\gcd(x,y)=1$.
The following statements hold.
\begin{enumerate}[label=(\roman*)]

\item \cite{luca09}
If $n$ is prime and $3\le D_i\le 99$, then the only
solutions $(D_i,x,y^n)$ to \eqref{theeq} are given by 
\begin{equation*}
(D_i,x,y^n) \in
\left\{
\begin{array}{ccc}
(3,37,7^3), &(11,31,3^5), &(19,559,5^7),\\
(59,7,3^3), &(59,21,5^3), &(59,525,41^3),\\
(59,28735,591^3), &(83,5,3^3), &(83,3785,153^3)
\end{array}
\right\}.
\end{equation*}

\item \cite{luca09} If $D_i>100$ factors as $7^a11^b$ or $7^a13^b$ for some
$a,b\ge 0$, then there are no solutions to \eqref{theeq}.

\item \cite{bugeaud01a, bugeaud01, arif02, bilu02}
If $n\ge 5$ is prime, $D_i > 100$ is squarefree, and $n$ does not divide 
$h(-D_i)$, then there are no solutions to \eqref{theeq}.

\item \cite{bugeaud01}
For fixed prime $y$ and odd $n$, and $100<D_i<2\sqrt{y^n}$, there is at most
one solution to \eqref{theeq}.
In such cases that there is a solution, we have
$$n < \frac{4\sqrt{D_i}}{\pi} \log(2e\sqrt{D_i}).$$
\end{enumerate}
\end{lem}
We note that in \cite{bugeaud01a} a list of infinite classes satisfying the
equation $\alpha x^2+\beta =\lambda y^n$, with $y$ a prime and
$\lambda\in\{1,\sqrt{2},2\}$ are given. We are fortunate that only the
infinite class ${\mathscr G}$ applies to our case, but even here the solutions
either have $n=1$ or $\beta> 2\sqrt{y^n}$, so that it, too, does not actually 
produce solutions that we are concerned with.
The results of \cite{bugeaud01} also provide infinite classes involving
the Fibonacci and Lucas sequences for $n=5$. However, for these classes
$\alpha\ne 1$ for all cases where $\beta>100$.
Though this is not made explicit in \cite{bugeaud01}, one can easily check this
is the case via the results in \cite{bilu01}.

\section{The case where $q$ is not a square}

We now return our attention to improving the upper bound for $V(f)$ in the case
where $q=y^n$ with $n\ge 3$ odd. As was noted at the end of \Cref{CostSection},
none of the $q$ to be considered in this section can be among the set $\X$ of
exceptions.

Throughout, we assume $f\in C_4(A)$.
For any integer $n$, we define $Q(n)$ to be the
greatest squarefree factor of $n$. For any prime $p$, we also set
$Q_p(n)$ to be the greatest squarefree, $p$-free factor of $n$. That is,
$Q_p(n)=Q(n)/\gcd(p,Q(n))$.
We begin with a useful extension of \Cref{eqlemma} (iii).
\begin{lem} \label{classnumlem}
Let $n\ge 5$ be odd and let $D_i>100$ be any $D_i$ for which
\eqref{theeq} has a solution.
Then $D_i$ is not squarefree or $Q_3(n)$ divides $h(-D_i)$.
\end{lem}
\begin{proof}
Suppose $D_i$ is squarefree.
In considering \eqref{theeq}, let $d\ne 3$ be any prime divisor of $n$, and
note that by hypothesis, we have a solution to the equation
\begin{equation} \label{summaryeq}
x^2+ D_i = 4y^n = 4(y^k)^d,
\end{equation}
where $n=kd.$ \Cref{eqlemma} (iii) says that since $D_i$ is squarefree, 
$d$ must divide $h(-D_i)$.
Since $d$ was any prime other than 3 that divides $n$, we conclude that
$Q_3(n)$ divides $h(-D_i)$, as claimed.
\end{proof}

We can now give our first improvement on the upper bound for general $q=y^n$,
$y,n>1$.
\begin{thm2}
Let $q=y^n$ for odd natural numbers $y,n>1$ and let $f\in C_4(A)$.
For $i=0,1,2,7,10,13,16,18$, set
$\beta_i=0,1,2,4,4,6,6,7$, respectively.
Then 
\begin{equation*}
V(f) \le q - \left(\frac{\sqrt{4q-D_i} - 1}{2}\right) - \beta_i, 
\end{equation*}
where
\begin{equation*}
i = 
\begin{cases}
0	&\text{ if $y^n=7^3$,}\\
1	&\text{ if $y^n=3^5$,}\\
2	&\text{ if $y^n=5^7$,}\\
7	&\text{ if $y^n\in \{3^3, 5^3, 41^3, 591^3\}$,}\\
10	&\text{ if $y^n=153^3$,}\\
13	&\text{ if $n=3^a$ for $a\ge 1$ and $y\notin\{3,5,7,41,153,591\}$,}\\
13	&\text{ if $n=3^a$ for $a\ge 2$ and $y\in\{3,5,7,41,153,591\}$,}\\
16	&\text{ if $n=5^a$ for $a\ge 1$ and $y\ne 3$,}\\
16	&\text{ if $n=5^a$ for $a\ge 2$ and $y=3$, and}\\
18	&\text{ otherwise.}\\
\end{cases}
\end{equation*}
This bound is tight in each of the exceptions given in the above list.
\end{thm2}
\begin{proof}
We apply the process outlined in \Cref{gensit}, and obtain the exceptions 
listed above via \Cref{eqlemma} (i) for $D_i<100$.
For $13\le i\le 17$, $D_i$ is squarefree, while the class number $h(-D_i)$ of
$\qf{D_i}$ is $3, 2, 2, 5, 3$, respectively.
To reduce to the claimed bounds in the remaining cases let $d$ be any 
prime divisor of $n$. We may apply \Cref{eqlemma} (iii) in conjunction with
\Cref{classnumlem} and get
\begin{equation*}
V(f) \le q - \left(\frac{\sqrt{4q-D_i} - 1}{2}\right) - B_{2i}.
\end{equation*}
We can then calculate $\beta_i=B_{2i}$ for each $i\in\{0,1,2,7,10,13,16\}$.
For each of these exceptions, the bound is tight, as we can directly construct
a function $f\in C_4(A)$ with the claimed image set size.
Finally, for $i=18$ we have $D_{18}=147,$ which is not squarefree, so that we may no longer
appeal to \Cref{eqlemma} for general $y$. One determines
$\beta_{18}=B_{36}=7$, and all claims are established.
\end{proof}
Finally, it is, perhaps, worth mentioning that the first $D_i$ which is not
squarefree and for which $h(-D_i)$ is divisible by an odd prime is
$D_{66}=531=3^2 \times 59$, where $h(-531)=3$. (This is because the first
$D_i$ for which $h(-D_i)=3$ is $D_7=59$.)
With much effort one could apply the standard techniques used in
\Cref{step0section} to reach the bound \eqref{gencosteq} for $i=66$ with a
list of additional exceptions. The law of diminishing returns, for both us and
the reader, dissuades us from doing this!

\section{The case where $q=p^n$ is not a square} \label{nonsquareprime}

For the remainder our goal is to exploit \Cref{eqlemma} (iv) to
improve the upper bound.
In particular, for the remainder we assume $q=y^n$ with $y=p$ a prime and $n$
odd.

We will need a slightly more useful (for our purposes) version of
\Cref{eqlemma} (iv).
While \cite{bugeaud01}, Theorem 2 only gives the bound stated in
\Cref{eqlemma} (iv), it's proof relies on Theorem 2 of Le \cite{le95}. 
In Le's result, there are two points we need to address. Firstly, while he
does not have a squarefree condition on $D_i$, he does require $D_i$ to not
be a square. Since $D_i\equiv3\bmod 8$, we are guaranteed of this. Additionally,
Le's proof makes clear that there is a minimal solution involved which
the solution for each exponent $n$ we might consider stems from. This minimal
solution has an exponent $z_1$ that necessarily divides $n$ and for which
$z_1|h(-4D_i)=h(-D_i)$. Thus, we can give the following slightly more exact 
version of \Cref{eqlemma} (iv).
\begin{lem} \label{betteriv}
For fixed prime $y$ and odd $n$, and $100<D_i<2\sqrt{y^n}$, there is at most
one solution to \eqref{theeq}. In such cases that there is a solution, we
have $\gcd(n,h(-D_i))>1$. In particular, if $\gcd(n,h(-D_i))=1$, then there is
no solution, and if $\gcd(n,h(-D_i))>1$, then we cannot say there is or is
not a solution.
\end{lem}
We now give a slightly refined version of the situation outlined in
\Cref{gensit}. Let $q=p^n$, with $p$ prime and $p,n\ge 3$ odd, 

\subsection*{Part A}

We begin needing to consider the equation
\begin{equation} \label{pgroupcase}
x^2 + D_i  = 4 p^n.
\end{equation}
Each time we get a solution to this, we get a largest preimage size of
$t=(1+x)/2$. That is the relation between these equations and our function $f$.
The larger the $D_i$, the smaller the $x$ relative to our $q$, and so the
smaller the $t$. 
However, one cannot necessarily deal with \eqref{pgroupcase} directly as there
may be a reduction involved when $p$ divides $D_i$.

\subsection*{Part B}

If $p|D_i$, then there can only be a solution to 
\Cref{pgroupcase} provided an even power of $p$ divides $D_i$, 
as otherwise we get a contradiction based on the divisibility of $x$ by $p$.
Let $p^{2l}$ be the largest even power of $p$ dividing $\gcd(D_i,q)$.
Set $E_i=D_i/p^{2l}$.
Since $p^{2l}\equiv 1\bmod 8$, we have $E_i=8j+3$ with $j\le i$; that is,
$E_i$ is in fact just some $D_j$ for some $j\le i$.
Now \Cref{pgroupcase} reduces to finding solutions to
\begin{equation} \label{Eeq}
x^2 + E_i  = 4 p^m,
\end{equation}
with $m=n-2l$ odd. 
If a solution $x$ to this equation is obtained, then our $t$ corresponding
to \eqref{pgroupcase} is $t=(1+p^l x)/2$ (in fact, $p^l x$ would
be a solution to the original equation involving $D_i$ and $q$).

Now, for $p>8$ we must have $2l\le (n-1)/2$, as
$D_i=8i+3$ and $2i<2\sqrt{q-1}$ by \Cref{sBounds}.
Thus,
\begin{equation*}
D_i = 8i +3 < 8\sqrt{q-1} + 3 < 8 p^{n/2} +3.
\end{equation*}
The largest power of $p$ less than this is $(n-1)/2$, as claimed.
In fact, if $n\equiv 3\bmod 4$, we actually have $2l\le (n-3)/2$.
Additionally, if $p<8$ is an odd prime, then $2l\le (n+1)/2$. We record this
in the following statement.
\begin{lem} \label{reductionlimit}
In the reduction process for \Cref{pgroupcase}, we can only reduce to
equations of the form \Cref{Eeq} with
\begin{equation*}
m \ge 
\begin{cases}
(n-1)/2 &\text{ if $p\in\{3,5,7\}$, and}\\
(n+1)/2 &\text{ otherwise.}\\
\end{cases}
\end{equation*}
\end{lem}
The key implication of \Cref{reductionlimit} is that the only reductions
that reach one of the cases in \Cref{eqlemma} (i) or \Cref{newgeneralbound}
occur when
\begin{itemize}
\item $n=7$ and $p\in\{3,5,7,41\}$, or
\item $n=11$ and $p\in\{3,5\}$, or
\item $n=13$ and $p=5$.
\end{itemize}
As will be seen, only the $n=7$ cases cause some complications, and so we 
treat that situation separately.

We note, in passing, that in the case where $q=y^n$ with $y$ not prime, the
reduction step dramatically complicates
the form of the resulting equation, and unless $y^2$ divides $D_i$, there
appears to be little way in which to control the equation in such a way that
we may get some sort of resolution. This is the reason that we restrict
ourselves to the prime power case.

\subsection{The Modified Class Divisor}



For any odd number $m\ge 3$, we define the {\em modified class divisor}
function $\MCD(m)$ to be the smallest $D_i$ for which
$\gcd(m,h(-D_i))>1$.
For example, according to OEIS, sequence A202084, $\MCD(3)=59$.
It is obvious that
$\MCD(m)$ is the minimum $\MCD(p)$ among all primes $p$ dividing $m$.
Computing using the Magma algebra package gives us the following table for
primes up to 31.
\begin{center}
\begin{tabular}{c|c}
Prime $m$	&$\MCD(m)$\\
\hline
3	&$D_7=59$\\
5	&$D_{16}=131$\\
7	&$D_{31}=251$\\
11	&$D_{82}=659$\\
13	&$D_{127}=1019$\\
17	&$D_{136}=1091$\\
19	&$D_{262}=2099$\\
23	&$D_{226}=1811$\\
29	&$D_{367}=2939$\\
31	&$D_{406}=3251$
\end{tabular}
\end{center}
Note that as $D_i\equiv 3\bmod 4$, $\MCD(m)$ must be squarefree, as otherwise
we could write $\MCD(m)=c^2 D_i$ for some smaller $D_i$, and 
$\qf{D_i}$ and $\qf{c^2D_i}$ are isomorphic, so that $\MCD(m)=D_i$ by the
definition.
We call $\MCD$ the modified class divisor only because we are insisting
on the value being $3\bmod 8$.

\subsection{A new bound for $V(f)$ when $q=p^n$ is not a square}

With those details behind us, we now prove a new 
upper bound for $V(f)$ when $f\in C_4(A)$ with $q$ a prime power.
\begin{thm2}
Let $q=p^n$ with $p$ an odd prime and where $n\ge 11$ is odd and
$\gcd(n,15)=1$. Additionally, assume 
$q\notin\{3^{11}, 5^{11}, 5^{13}\}$.
If $f\in C_4(A)$, then
\begin{equation*}
V(f)
\le q + 1 - \left(\frac{\sqrt{4q-D_j} + \sqrt{4j+1}}{2}\right),
\end{equation*}
where 
\begin{equation*}
D_j =
\begin{cases}
\Min\left\{\MCD(n),\MCD(n-2)p^2,107p^4\right\} &\text{ if $n\equiv 1\bmod 3$, or}\\
\Min\left\{\MCD(n),107p^2\right\} &\text{ if $n\equiv 2\bmod 3$.}
\end{cases}
\end{equation*}
\end{thm2}
\begin{proof}
From the processes outlined in \Cref{gensit} and at the start of this
section, we wish to determine the smallest integer $D\equiv 3\bmod 8$ for
which we can no longer preclude the possibility that the equation
$x^2+D=4p^n$ has a solution.

Set $S=\{p^{2l}\MCD(n-2l)\,:\, 0\le 2l\le (n-1)/2\}$.
Thanks to our hypotheses,
it follows from \Cref{reductionlimit} and the noted key implication of it given
straight after, that the set $S$ corresponds
to the smallest $D_i$ for each reduction possibility and where
$\gcd(n,h(-D_i)) > 1$.
Choose any $D_s\in S$ and consider the equation $x^2+D_s = 4p^n$.
By construction, this equation allows for a reduction by $p^{2l}$, so that we
must be in Part A with $E_s=\MCD(n)$ when $l=0$, or in Part B with
$E_s=\MCD(n-2l)$ when $l>0$.
Since each $\MCD(n-2l)$ is squarefree, we know that
there cannot be a further reduction in the equation. Additionally,
by the definition of the modified
class divisor, $\gcd(n-2l,h(-4E_s))>1$. Thus, \Cref{betteriv} does not 
preclude the possibility that $x^2+D_s=4p^n$ has a solution.

Suppose that there is a solution to $x^2+D=4p^n$ for some $D<\Min(S)=D_{min}$.
If $p^2$ does not divide $D$, then we have no reduction and we are in Part A.
Now $D<D_{min}\le\MCD(n)$, so that there can be no solution in this case by
the definition of the modified class divisor and \Cref{betteriv}.
Otherwise, we must have a reduction by $p^{2l}$ for some $l\ge 1$.
This puts us in Part B with $D=p^{2l} E$, and considering \eqref{Eeq}
with $m=n-2l$. Since
$$p^{2l} E = D <  D_{min}\le p^{2l} \MCD(m),$$
we see $E<\MCD(m)$. \Cref{betteriv} now guarantees there is no solution
to $x^2+E=4p^m$, and so no solution to $x^2+D=4p^n$, which is a contradiction.

We now prove the bounds for $n\equiv 1\bmod 3$, the case
$n\equiv 2\bmod 3$ following similarly. Let
\begin{equation*}
S' = \left\{\MCD(n),\MCD(n-2)p^2,59p^4\right\}.
\end{equation*}
Note that $S'\subseteq S$ and $59p^4\in S'$ as $3|(n-4)$.
We claim that $\Min(S')=\Min(S)$.
Let $D_s=p^{2l}\MCD(n-2l)$ for some $l\ge 3$. Since $\MCD(n-2l)\ge \MCD(3)$,
we see $D_s>p^4\MCD(n-4)=59p^4$. Thus any element of $S\setminus S'$ is
larger than $59p^4\in S'$, which proves the claim.

Now, if $\Min(S)=59p^4$, then the resulting reduced equation is
$x^2+59=4p^{n-4}$.
However, under our hypotheses, it can be seen from an inspection of
\Cref{eqlemma} (i) that this equation does not have a solution (renember we
never reduce to any of the cases listed there).
Consequently, the first $D_i$ we need be concerned
with for the $p^{n-4}$ case will be the first $D_i>100$ for which
$3|h(-D_i)$, which is $D_{13}=107$. 
We therefore revisit $S'$ and consider instead
\begin{equation*}
S'' = \left\{\MCD(n),\MCD(n-2)p^2,107p^4\right\}.
\end{equation*}
As before, we claim $\min(S'')=\min(S)$. It can only fail to hold if a
larger reduction produces a smaller value for $p^{2l}\MCD(n-2l)$, but since
$107< \MCD(5)$ and $59p^6>107$, it is easily seen that this never happens.
This proves the theorem for $n\equiv 1\bmod 3$, and the $n\equiv 2\bmod 3$ case
follows in almost the same manner.
\end{proof}
Thanks to \Cref{nonTriCost}, we can compute directly the exact upper bounds
for $V(f)$ when $q\in\{3^{11},5^{11},5^{13}\}$. Indeed, we have 
\begin{equation*}
V(f)\le q - \cost(q-1) = q -
\begin{cases}
441 &\text{ if $q=3^{11}$,}\\
7054 &\text{ if $q=5^{11}$,}\\
35013 &\text{ if $q=5^{13}$.}
\end{cases}
\end{equation*}
As an illustration of the quality of the bound given in the theorem, for
$q=41^{11}$ we can directly calculate that there exists an $f\in C_4(A)$ for which
$$V(f)= q-\cost(q-1) = q - 741858080,$$
which is only 16423 better than the upper bound given by \Cref{secondnewbound}.

For $n=7$, we may proceed as above, except that we can only have the reduction
by $p^2$, a $p^4$ reduction being disallowed via \Cref{sBounds}.
Since $\MCD(7)=251<131 p^2=p^2\MCD(5)$ for any prime $p$, we calculate 
$B_{502}=29$ and get the following bound.
\begin{thm} \label{boundforn=7}
Let $q=p^7$ with $p\ge 7$ a prime.
If $f\in C_4(A)$, then
\begin{equation*}
V(f) \le q - 29  -\left(\frac{\sqrt{4q-251} -1}{2}\right).
\end{equation*}
\end{thm}
Finally, we note that for $q=p^n$ with $\gcd(n,15)>1$, the same methods yield
no better a bound than that stemming from one of the specific $i\ne 10$ cases of
\Cref{newgeneralbound}.

We end this section by listing a couple of examples to illustrate the bound of
\Cref{secondnewbound}.
All calculations below were carried out using the Magma algebra package.
\begin{itemize}
\item For $n=11$, we see that since $\MCD(11)/\MCD(3)<12<p^2$ for any prime
$p\ge 5$, and so $D_j=\MCD(11)=659$ in \eqref{finalnewbound}.

\item For $n=29$, we see that $\MCD(29)=2939$, while
$\MCD(27)=\MCD(3)=59$, so that in this case the $D_j$ in the bound resolves to
$D_j=59p^2$ for $p\le 7$.

\item For $n=127$, $\MCD(127)=D_{8056}=64451$, while $\MCD(125)=\MCD(5)=251$.
Thus, $D_j=251p^2$ for $p\le 13$.
\end{itemize}

\section*{End note}

We feel it appropriate to end this paper with a couple of comments.

Our primary motivation
for studying upper bounds of $V(f)$ remains our desire to see how far 
planar functions must be from being permutations. Recent work by Chen and
Coulter \cite{coulter25} shows how non-bijective functions with low
differential uniformity (and planar functions have optimally low) can be
altered to create bijections with low differential uniformity. Such functions
are important in information security. 
In this paper we have effectively resolved the upper bound problem for 
functions $f\in C_4(A)$. However, this condition is far weaker than 
either $f\in C_3(G)$, as appears in the connection to a projective plane
of order 18 in \Cref{thm1} (iii), or $f$ being planar.
We strongly believe that these bounds are not tight for either class functions,
but new methods need to be found that effectively use the stronger conditions
functions in either class must satisfy.

And finally, all of our results concerning $V(f)$ for functions $f$ on $A$
could be extended to functions $f:A\rightarrow B$ for  arbitrary $B$, though
the tightness of the bound may fail. For example, if the cardinality of $B$
is small relative to the cardinality of $A$, then $N_2$ may be forced to be
large, and consequently our bound would say no more than $V(f)\le \#B$.

\section*{Statements and Declarations}

\subsection*{Financial Support}
Partial financial support was received from a bequest of the Estate
of Francisco Javier ``Pancho" Sayas.

\subsection*{Conflict of Interest Statement}
Neither author has any competing interest to declare that is relevant to the
contents of this article.

\subsection*{Associated Data}
This manuscript has no associated data beyond that supplied in the
appendix below.


\providecommand{\bysame}{\leavevmode\hbox to3em{\hrulefill}\thinspace}
\providecommand{\MR}{\relax\ifhmode\unskip\space\fi MR }
\providecommand{\MRhref}[2]{%
  \href{http://www.ams.org/mathscinet-getitem?mr=#1}{#2}
}
\providecommand{\href}[2]{#2}

\section{Appendix: Costs of the exceptions} 

As we saw in \Cref{nonTriCost}, there are 51 exceptions where the greedy
approach of taking the maximum possible size of an equivalence class does not
generate, or at least uniquely generate, the minimum cost. In this appendix, we
give the exact costs for each of the exceptions, and the ways in which they
can be achieved.

Before doing so, a note on the notation: the way in which a minimum cost can
be achieved for an equivalence relation $\sim$ with $||\sim||=n$ is written
as a non-increasing sequence $[n_1,n_2,\ldots,n_k]$, where each $n_i$ is the
cardinality of the next largest equivalence class. Thus, we have the two
identities $\cost(n)=\sum_{i=}^k n_i-1$ and $n-1=\sum_{i=1}^k n_i(n_i-1)$.
The $n$ given in boldface are those where $\cost(n)$ cannot be achieved
with an equivalence class of maximum possible size.

\begin{center}
\begin{tabular}{||c|c|c|c||c|c|c|c||}
\hline
$n$	&$\cost(n)$	&$\rhombicfl{n}$	&Ways obtained
&$n$	&$\cost(n)$	&$\rhombicfl{n}$	&Ways obtained\\
\hline
24	&6	&5	&[5,2,2],\ [4,4]
&1254	&44	&35	&[35,8,3,2],\ [34,12]\\
{\bf 40}	&8	&6	&[5,5]
&1256	&45	&35	&[35,8,3,2],\ [34,12,2]\\
50	&9	&7	&[7,3,2],\ [5,4]
&1692	&50	&41	&[41,7,3,2,2],\ [40,12]\\ 
52	&10	&7	&[7,3,2,2],\ [5,4,2]
&1962	&54	&44	&[44,8,4,2],\ [43,13]\\
86	&12	&9	&[9,4,2],\ [8,6]
&2136	&56	&46	&[46,8,3,2,2],\ [45,13]\\
88	&13	&9	&[9,4,2,2],\ [8,6,2]
&{\bf 2344}	&59	&48	&[47,14]\\
120	&14	&11	&[11,3,2,2],\ [10,6]
&2438	&60	&49	&[49,9,4,2],\ [48,14]\\
128	&15	&11	&[11,4,2,2],\ [9,8]
&2440	&61	&49	&[49,9,4,2,2],\ [48,14,2]\\
174	&17	&13	&[13,4,3],\ [12,7]
&2632	&62	&51	&[51,9,3,2,2],\ [50,14]\\
180	&18	&13	&[13,4,4],\ [10,10]
&2966	&66	&54	&[54,10,4,2],\ [53,15]\\
198	&18	&14	&[14,4,2,2],\ [13,7]
&2968	&67	&54	&[54,10,4,2,2],\ [53,15,2]\\
{\bf 238}	&20	&15	&[14,8]
&3180	&68	&56	&[56,10,3,2,2],\ [55,15]\\
266	&21	&16	&[16,5,3],\ [15,8]
&3780	&74	&61	&[61,11,3,2,2],\ [60,16]\\
268	&22	&16	&[16,5,3,2],\ [15,8,2]
&3870	&74	&62	&[62,9,4,2,2],\ [61,15]\\
296	&22	&17	&[17,4,4],\ [16,8]
&4530	&80	&67	&[67,10,4,3],\ [66,16]\\
300	&23	&17	&[17,5,3,2],\ [15,10]
&4662	&81	&68	&[68,10,4,2,2],\ [67,16]\\
378	&25	&19	&[19,6,3],\ [18,9]
&5384	&87	&73	&[73,11,4,3],\ [72,17]\\
414	&26	&20	&[20,6,2,2],\ [19,9]
&5528	&88	&74	&[74,11,4,2,2],\ [73,17]\\
534	&29	&23	&[23,5,3,2],\ [22,9]
&6312	&94	&79	&[79,12,4,3],\ [78,18]\\
690	&33	&26	&[26,5,5],\ [25,10]
&6468	&95	&80	&[80,12,4,2,2],\ [79,18]\\
740	&34	&27	&[27,6,3,2],\ [26,10]
&10280	&118	&101	&[101,13,4,2],\ [100,20]\\
866	&37	&29	&[29,7,4],\ [28,11]
&11550	&125	&107	&[107,14,5,3],\ [106,21]\\
{\bf 922}	&38	&30	&[29,11]
&11762	&126	&108	&[108,14,4,4],\ [107,21]\\
980	&39	&31	&[31,7,3,2],\ [30,11]
&20574	&164	&143	&[143,16,5,3,2],\ [142,24]\\
982	&40	&31	&[31,7,3,2,2],\ [30,11,2]
&22950	&173	&151	&[151,17,5,3,2],\ [150,25]\\
{\bf 1188}	&43	&34	&[33,12]
&	&	&	&\\
\hline
\end{tabular}
\end{center}

\end{document}